\documentclass[11pt,fleqn]{amsart} 

\usepackage[usenames,dvipsnames,condensed]{xcolor}

\usepackage{amsthm}
\usepackage{amsfonts}
\usepackage[english]{babel}
\usepackage[usenames]{xcolor}
\usepackage{graphicx}
\usepackage{soul}
\usepackage{stfloats}
\usepackage{morefloats}
\usepackage{cite}
\usepackage{lscape}
\usepackage{epstopdf}
\usepackage{bbm}
\usepackage[lite]{amsrefs}
\usepackage{tikz-cd}
\usepackage{tikz}
\usepackage{shuffle}
\usepackage{lscape}
\usepackage{rotating}
\usepackage{amsmath,amstext,amsopn,amsfonts,eucal,amssymb}
\usepackage{graphicx,wrapfig,url}

\newcommand\N{{\mathbb N}}
\newcommand\Z{{\mathbb Z}}

\newcommand\C{{\mathbb C}}

\newtheorem{theorem}{Theorem}[section]

\newtheorem{lemma}[theorem]{Lemma}
\newtheorem{proposition}[theorem]{Proposition}
\newtheorem{definition}[theorem]{Definition}

\newtheorem{example}[theorem]{Example}

\newtheorem{remark}[theorem]{Remark}

\begin{document}
	\title[Quantum invariants from TSD cohomology]{Quantum invariants of framed links from ternary self-distributive cohomology}
	\author{Emanuele Zappala} 
	\address{Institute of Mathematics and Statistics, University of Tartu\\
	Narva mnt 18, 51009 Tartu, Estonia} 
	\email{emanuele.amedeo.zappala@ut.ee \\ zae@usf.edu}
	
	\maketitle
		
			\begin{abstract}
				The ribbon cocycle invariant is defined by means of a partition function using ternary cohomology of self-distributive structures (TSD) and colorings of ribbon diagrams of a framed link, following the same paradigm introduced by Carter, Jelsovsky, Kamada, Langfor and Saito in Transactions of the American Mathematical Society 2003;355(10):3947-89, for the quandle cocycle invariant. In this article we show that the ribbon cocycle invariant is a quantum invariant. We do so by constructing a ribbon category from a TSD set whose twisting and braiding morphisms entail a given TSD $2$-cocycle. Then we show that the quantum invariant naturally associated to this braided category coincides with the cocycle invariant. We generalize this construction to symmetric monoidal categories and provide classes of examples obtained from Hopf monoids and Lie algebras. We further introduce examples from Hopf-Frobenius algebras, objects studied in quantum computing. 
			\end{abstract}
		
		\tableofcontents
		
	\section{Introduction}
	
	Self-distributivity of binary operations is well known to be an algebraic formulation of the Reidemeister move III in knot theory. Sets with self-distributive operations (i.e. shelves) satisfying extra conditions encoding Reidemeister moves I and II have been used, starting in the 1980's, to construct invariants of knots and links. For instance, Joyce and Matveev independently defined what is now known as the {\it fundamental quandle} of a knot \cite{Joy,Mat}, whose construction is given as a presentation where the generators correspond to the arcs of a knot diagram, while the relations formally correspond to the conjugation operation in a group. Shelves satisfying the algebraic Reidemeister move II condition are called racks, while those satisfying also the algebraic counterpart of Reidemeister move I are called quandles. 
	
	More recently, the notion of (co)homology of quandle has been introduced, and a state-sum invariant of links that utilizes quandle cohomology has been constructed in \cite{CJKLS}. The resulting ``cocycle invariant''  is obtined as a sum over all the colorings of a knot diagram, the states, of all the products of Boltzmann weights, determined by quandle  $2$-cocycles. Although computing cocycle invariants introduces a new problem, that of obtaining nontrivial quandle second cohomology classes, it is in general easier to compare two cocycle invariants rather than comparing the fundamental quandle of two knots. 
	
	Moreover, it is known that quandles induce solutions to the set-theoretic Yang-Baxter equation and therefore, upon linearizing the corresponding set-theoretic map, they produce Yang-Baxter operators \cite{Eis}. In fact, given a quandle and a $2$-cocycle $\alpha$, one can construct a Yetter-Drinfel'd module (i.e. a particular instance of a ribbon category) \cite{Gra} and, consequently, one can obtain quantum link invariants associated to the ribbon category following a standard procedure as in \cite{Tur}. It naturally arises the question of whether the two types of invariant are somehow related. A positive answer has been given in \cite{Gra}, where it is shown that the invariants coincide in a suitable sense. 
	
	Ternary self-distributive structures are generalizations of binary shelves to the setting of ternary operations. A suitable diagrammatic interpretation of crossing of ribbons in terms of ternary operations translates the fundamental moves for the isotopy equivalence of framed links into a ternary analogue of rack. A corresponding state-sum invariant that uses cohomology of ternary racks and colorings of ribbon diagrams associated to  framed links is then constructed \cite{EZ} following the same reasoning as in the binary case. This invariant, called {\it ribbon cocycle invarinat}, has been studied for a fundamental class of ternary racks, called group heaps, and it has been seen to detect nontrivial framing of links \cite{SZ}.
	
	On the other hand, group heaps can be generalized to certain structures, named {\it quantum heaps}, that naturally arise from involutory Hopf algebras, i.e. having antipode that squares to the identity map. A corresponding construction for Hopf monoids in symmetric monoidal categories exists \cite{heap}, providing a large class of examples for ternary self-distributive objects in symmetric monoidal categories, in the sense of \cite{ESZ} Section~8. It is therefore possible to develop an analogue of the ternary set-theoretic theory in symmetric monoidal categories.
	
	The scope of this article is that of using ternary self-distributive (TSD) structures and their ternary cocycles to construct ribbon categories whose twisting morphisms are nontrivial, and study the corresponding link invariants. The starting point of this study follows the paradigm that has been used in \cite{Gra} to prove that the cocycle invariants are indeed quantum invariants. We prove, in fact, that set-theoretic TSD structures and a choice of a ternary $2$-cocycle are linearized to obtain a braiding in a suitably constructed symmetric monoidal category. The construction is similar to that of the {\it braid category} \cite{FY,Kas}, but braiding and twisting are induced by the TSD structure following the doubling functorial procedure in \cite{ESZ}, and using TSD cocycles to twist the morphisms obtained.   Analogously to the fact that the ribbon cocycle invariant detects nontrivial framings \cite{SZ}, we obtain that the twisting defined in this category is nontrivial, as opposed to the case of Yetter-Drinfel'd modules associated to (binary) quandle operations. 
	
    On the one hand, there is no strict need of defining a ribbon category out of the data of a TSD and a ternary $2$-cocycle, in the sense that we can obtain a representation of the framed braid group in a similar fashion as in \cite{Tur}, from which a corresponding quantum invariant would naturally arise. 
    On the other hand, though, this construction easily generalises to multiple objects where ``self-distributive'' ternary actions are defined. These produce a more general family of ribbon categories where the twists are obtained by TSD operations as in the previous case, while the braidings are obtained from ternary actions. Moreover, the braiding and twisting morphisms can be deformed by cohomological classes that twist the weights and entail the operations and mutual actions of the underlying structure. Among the examples that we present in this paper, we find mutually distributive structures and their labeled cohomology, whose algebraic properties were studied in \cite{ESZ}, and $G$-families of quandles and their cohomology theory, extensively studied in connection with knotted handlebody invariants \cite{IIJO,Nos}. 
    
    The approach mentioned above, in addition, is particularly suitable to be generalized to the case of TSD objects in symmetric monoidal categories. As observed above, in fact, the notion of heap has a counterpart obtained from involutory Hopf monoids in symmetric monoidal categories, therefore providing a fertile ground for a general theory that associates a ribbon category to a symmetric monoidal category along with a TSD object in it. Using the TSD morphism we obtain, in fact, a Yang-Baxter operator in the tensor product of the TSD object we start with, and use this to define the braiding of the ribbon category. The twist is obtained via the same procuedure by interpreting twists as self-intersections of a ribbon. In other words, Reidemeister move I does not hold when we consider framed links, but it is replaced by a twisting which can be defined using a variation of the braiding.
    
    We have mentioned that we utilize TSD cohomology classes to deform the braiding and twisting in the case of linearized TSD operations. When working in a symmetric monoidal category, we can introduce a categorical version of the $2$-cocycle condition. The setting, here, generalizes the set-theoretic one in two fundamental ways. Recall the set-theoretic $2$-cocycle condition, which reads
    $$
    \psi(x,y,z) - \psi([x,u,v],[y,u,v],[z,u,v]) - \psi(x,u,v) + \psi([x,y,z],u,v) = 0
    $$
    for all $x,y,z,u,v\in X$, where $(X,[-, -, -])$ is a TSD set. Firstly, observe that certain elements appear in more than one term, and therefore are repeated. This is no particular concern when dealing with set-theoretic structures, but in a general symmetric monoidal category, it is required that each instance of a repreated element is replaced by an instance of a comultiplication morphism. In fact, the definition of TSD object, see for instace Section~8 in \cite{ESZ} for $n$-ary case, implements this perspective already, and it is somehow natural to expect that it carries on to the $2$-cocycle condition. Secondly, in the set-theoretic case coefficients of cohomology are taken in a group, and linearization naturally requires the coefficients to be represented in the ground field. In an arbitrary symmetric monoidal category, we interpret this situation as an equality holding in the unit object of the category. The object of coefficients naturally acts on the TSD object allowing the ``cocycle'' to perturb the Yang-Baxter operator associated to the TSD morphism. If one thinks of the group algebra associated to a group as being a Hopf algebra where the comultiplication is simply the splitting of an element in two identical copies, then the categorical interpretation of the $2$-cocycle condition seems to be on the same footing as the $2$-cocycle condition in the ground field of the linearization of a set-theoretic operation. 
    
    In the general situation, one further assumption is necessary, in order to apply the same construction as in the category of vector spaces. Namely, one needs to assume that the category is $\mathbb I$-linear, where $\mathbb I$ is the unit object. Then the $2$-cocycles are assumed to take values in the ground object $\mathbb I$ and, moreover, they are supposed to satisfy a convolution inversion formula, in order to allow the definition of inverses. This is naturally satisfied in the linearized case, since comultiplication is simply diagonal, and coefficients in a group are automatically invertible.
    
    Naturally, as in the set-theoretic case one can obtain ribbon categories from multiple TSD sets having suitable ternary actions and families of ternary $2$-cocycles, we can generalize the previous construction in a symmetric monoidal category where multiple TSD objects along with certain ternary morphisms are defined. An interesting class of examples arises from ternary augmented racks, where the axioms of augmentation can be easily translated from the case of vector spaces and Hopf algebras to that of Hopf monoids in a symmetric monoidal category. 
    
    \subsection{Main results}
    
    We proceed now to concisely summarize the main results of the present article. 
     
     The first result (Theorem~\ref{thm:ribboncat}) is that starting from a TSD set $(X,T)$ and a given ternary $2$-cocycle $\alpha\in Z^2(X,A)$ with coefficients in an abelian group $A$, we construct a ribbon category $\mathcal R^*_\alpha(X)$ whose braidings are constructed out of a Yang-Baxter operator arising from $(X,T)$ and deformed by the cocycle $\alpha$. Moreover, it is shown that the ribbon category is well-defined, up to equivalence of braided categories, with respect to the cohomology class of $\alpha$, in the sense that if $\beta$ represents the same cohomology class, then there exists a braided functor that gives an equivalence of categories between $\mathcal R^*_\alpha(X)$ and $\mathcal R^*_\beta(X)$. Moreover, a similar construction is shown to hold when starting with a family of TSD sets $\{X_i\}_{i\in I}$ along with maps $T_{ij}: X_i\times X_j\times X_j \longrightarrow X_i$ satisfying a generalized version of TSD condition (Theorem~\ref{thm:compatibleribboncat}). In this situation, we define the notion of ternary $2$-cocycles for the family $\{X_i\}_{i\in I}$ and use them to deform the Yang-Baxter operator associated to it. We therefore construct a braiding and a twisting in order to obtain a ribbon category whose families of objects and morphisms are larger than those of $\mathcal R^*_\alpha(X)$. In Theorem~\ref{thm:quantum} we show that the quantum invariants associated to the ribbon category $\mathcal R^*_\alpha(X)$ coincide with the (state-sum) ribbon cocycle invariant in a suitable sense, i.e. when we take a representation of the coefficient abelian group of cohomology in the ground field. This construction is generalized to the setting of TSD objects in symmetric monoidal categories, of which TSD sets in the category of sets are a particular instance. It is shown that in this setup, from a TSD object we obtain a Yang-Baxter operator which is deformed by means of what is hereby called a categorical $2$-cocycle. The new Yang-Baxter operator is used then to construct the braiding in what is shown to be a ribbon category (Theorem~\ref{thm:ribbongeneral}).
    
     \subsection{Organization of the article}
     
     This article is structured as follows. We review some preliminary material in Section~\ref{sec:pre}, where we recall binary and ternary self-distributive structures, the cocycle invariant and some basic notions regarding symmetric monoidal categories. In Section~\ref{sec:ribboncocy} we give a detailed account of the construction of the ribbon invariant of framed links as well as a proof of its being well-posed. In Section~\ref{sec:ribboncat} we show that starting from the data of a TSD set and a ternary $2$-cocycle, there exists a ribbon category determined up to equivalence of categories with respect to cohomology class of the $2$-cocycle. Moreover, it is shown that a similar construct exists starting from a family of TSD structures with some extra compatibility conditions and an analogue of the notion of ternary $2$-cocycle. The corresponding ribbon cateogry has a wider class of objects and morpshims with respect to the previous one. We then proceed to show, in Section~\ref{sec:ribbonquantum}, that the (state-sum) ribbon cocycle invariant coincides with the quantum invariant associated to the ribbon category arising in Section~\ref{sec:ribboncat}. Section~\ref{sec:examples} presents various examples to elucidate the construction in practice. Section~\ref{sec:generalized} is devoted to generalizing the theory developed in the previous sections in the context of symmetric monoidal cateories and TSD objects. The notion of categorical $2$-cocycle condition is introduced in order to deform the braidings obtained from TSD objects, in a fashion that follows the paradigm of Section~\ref{sec:ribboncat}. Quantum invariants associated to this class of ribbon categories are discussed in Section~\ref{sec:quantumgeneral}. Finally, further examples arising from ternary racks are given in the Appendix.
     
	\vspace{1cm}
	\noindent
	
	{\noindent{\bf Acknowledgements.} This research has been funded by the Estonian Research Council under the grant: MOBJD679. The author is grateful to M. Elhamdadi and M. Saito for useful conversations.}
	
	\section{Preliminaries}\label{sec:pre}
	
	In this section we provide preliminary material that is used throughout the article. 
	
	\subsection{Racks, quandles and cocycle invariants}
	\begin{sloppypar}
	Racks are (non-associative) magmas satisfying the self-distributive property given by $(x*y)*z = (x*z)*(y*z)$ for all $x,y,z$, such that the right multiplication maps are bijections. Self-distributivity is an ``algebraization'' of the topological notion of Reidemeister move III, while the requirement that right multiplications be bujective corresponds to imposing Reidemeister move II. Idempotent racks are called quandles, where idempotence corresponds to the remaining Reidemeister move I. It is well known that knot and link isotopy classes in $\mathbb R^3$, or $\mathbb S^3$, can be characterized combinatorially via their diagrams, i.e. projections on the plane satisfying certain regularity properties, and Reidemeister moves I, II and III. Consequently, quandles have been used in \cite{CJKLS} to construct state-sum invariants of links, named cocycle invariants. Fundamental roles in the definition and validity of the cocycle invariant are played by the notion of quandle coloring of a knot/link diagram, and a cohomology theory associated to racks and quandles. In fact, loosely speaking, the invariant is defined by considering all possible colorings of a fixed given diagram of a link $\mathcal L$, and multiplying the weights of each crossing of the diagram, each of which defined by applying a pre-determined $2$-cocycle to the colors meeting at the crossing. When applying any of the Reidemeister moves to pass from one diagram of $\mathcal L$ to the other, i.e. when performing an isotopy on $\mathcal L$, the colors of the diagrams correspond bijectively by virtue of the axioms defining a quandle, and the weights remain unchanged because of the definition of quandle cohomology. 
	\end{sloppypar}
	
	We proceed to briefly review the notion of quandle coloring of a link diagram, and the definition of cohomology associated to a quandle $Q$. A reference for both definitions is the article \cite{CJKLS}, where the cohomology utilizes abelian coefficients, while the case with non-abelian coefficients is treated in \cite{CEGS}. Let $\mathcal L$ be an oriented link, let $\mathcal D$ indicate an oriented diagram of $\mathcal L$, and let $Q$ be a quandle, with operation denoted by the symbol $*$. A coloring of $\mathcal D$ by $Q$ is a map $\mathcal C: R \longrightarrow Q$, where $R$ denotes the set of arcs of the diagram $\mathcal D$, satisfying the conditions given in Figure~\ref{fig:qcoloring}, for positive and negative crossings. 
	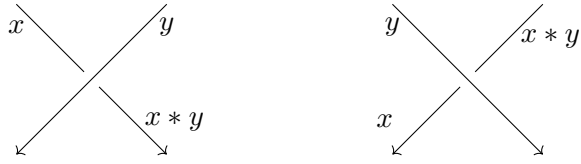
\begin{figure}
		\begin{center}
			\begin{tikzpicture}
			\draw[->] (2,2) -- (0,0);
			\draw (0,2)--(0.9,1.1);
			\draw[->] (1.1,0.9) -- (2,0);
			
			\draw[->] (5,2)--(7,0); 
			\draw (7,2)--(6.1,1.1);
			\draw[->] (5.9,0.9)--(5,0);
		    
		    \node at (0,1.7) {$x$};
		    \node  at (2.1,0.45) {$x*y$};
		    \node at (2,1.7) {$y$};
		    
		    \node  at (4.9,0.45) {$x$};
		     \node  at  (5,1.7) {$y$};
		     \node at (7.1,1.55) {$x*y$};
			\end{tikzpicture}
		\end{center}
	\caption{Coloring condition for positive crossings (left) and negative crossings (right).}
	\label{fig:qcoloring}
	\end{figure}
	
	Let $Q$ be a quandle and define chain groups $C_n(Q)$ to be the free abelian group generated by the elements of $Q^n$ for each $n$. Then, we define the $n^{\rm th}$-differential $\partial_n$ on generators according to the assignment
	\begin{eqnarray*}
	\lefteqn{\partial_n (x_1,\ldots , x_n)}\\
	&=& \sum_{i=2}^n (-1)^n[(x_1,\ldots ,x_{i-1},\hat x_i, x_{i+1}, \dots , x_n)\\
	&& \hspace{0.5cm} - (x_1*x_i, \ldots, x_{i-1}*x_i, \hat x_i,x_{i+1}, \ldots , x_n)]
	\end{eqnarray*}
where we have used $\hat{}$ to indicate omission of an element. Observe that the first term in the sum is the ``usual'' simplicial term, while the second term contains the information associated to the operation $*$, determining the quandle structure. One proves directly that the maps $\partial_n$ satisfy the pre-simplicial conditions and it follows automatically that $\partial_{n-1}\circ\partial_n = 0$, from which we have a well defined chain complex and an associated homology theory called {\it rack homology}. {\it Quandle homology} is obtained by quotienting out the sub-complex $C^{\rm q}_n(Q)$ generated as a free abelian group by $n$-tuples of $Q^n$ where $x_i = x_{i+1}$ for some $i$. In fact, it is the indempotency condition that induces well defined maps $\partial_n$ when restricting on the subgroups $C^{\rm q}_n(Q)$. Taking $A$ to be an abelian group and dualizing the rack and quandle chain complexes, one obtains associated cohomology theories which we denote by $H^n(Q;A)$ and $H^n_{\rm q}(Q;A)$, respectively. A representative $\phi$ of a second cohomology class $[\phi]\in H^2_{\rm q}(Q;A)$ satisfies the $2$-cocycle condition, which takes the form
$$
\phi(x,y) - \phi(x*z,y*z) - \phi(x,z) + \phi(x*y,z) = 0,
$$
for all $x,y,z\in Q$. The $2$-cocycle condition is related via a diagrammatic interpretation to Reidemeister move III, as depicted in Figure~1 in \cite{CJKLS}, while the $3$-cocycle condition, which we do not explicitly consider herein, is related to the tetrahedron move and shadow colorings. 

Fix now a coloring $\mathcal C$ by a quandle $Q$, defining at each crossing $\tau_i$ of Figure~\ref{fig:qcoloring} a Boltzmann sum, $\mathcal B(\tau_i,\mathcal C)$, as $\psi(x,y)$, for positive crossing (left panel), and $\psi(x,y)^{-1}$ for negative crossing (right panel), where $\psi\in Z^2(Q,A)$ is a quandle $2$-cocycle, one defines the state-sum (or partition function) 
$$
\sum_{\mathcal C} \prod_i \mathcal B(\tau_i,\mathcal C)
$$ 
for any given diagram of a knot, and where the sum runs over all the colorings $\mathcal C$ of the fixed diagram, and the product runs over all the crossings. This state-sum, called {\it cocycle invariant}, is shown to be an invariant of knots in \cite{CJKLS}, where it has been firstly introduced. When dealing with a link, one proceeds analogously for each component and defines an invariant that is a vector with as many entries as the components of the given link. 


\subsection{Framed links and their diagrams}
Framed links are embeddings of finitely many copies of $\mathbb S^1\times D^2$, i.e. solid tori, in the three dimensional space $\mathbb R^3$, or its compactification $\mathbb S^3$. Alternatively, framed links can be defined as links along with a choice of a section of their normal bundle. Diagrammatically, a frmed link $\mathcal L$ is represented by a link diagram whose arcs are thickened to be ribbons. This thickening is obtained in a standard way by doubling each arc so to obtain a second copy of the link diagram, parallel to the first one. The one lets the second copy mimick the over/under passing information of the first diagram. Such a thickened diagram is called {\it blackboard framing}. A crossing of a diagram whose arcs have been thickened into a ribbon is represented in Figure~\ref{fig:ribboncrossing}. From the figure is clear that the coloring paradigm corresponding to that of quandles changes. We can think of each crossing as two arcs, each of them underpassing two arcs. The coloring rule is suitably defined by means of ternary racks. This concept, introduced to the author by M. Saito, is formalized in Section~\ref{sec:ribboncocy}, where it is also given a construction of the ribbon cocycle invariant. 

\begin{figure}
	\begin{center}
			\begin{tikzpicture}
		\draw (0,3) -- (3,0);
		\draw (0.3,3) -- (3.3,0);
		
		\draw (3,3) -- (1.8,1.8);
		\draw (1.3,1.3) -- (0,0);
		
		\draw (3.3,3) -- (2,1.7);
		\draw (1.5,1.2) -- (0.3,0);
		
		\draw (0,3) -- (0.3,3);
		\draw (3,3) -- (3.3,3);
		\draw (0,0) -- (0.3,0);
		\draw (3,0) -- (3.3,0);
		
		\end{tikzpicture}
	\end{center}
\caption{Crossing of a blackboard framing of a framed link.}
\label{fig:ribboncrossing}
\end{figure}
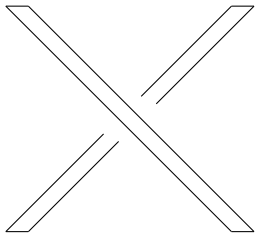

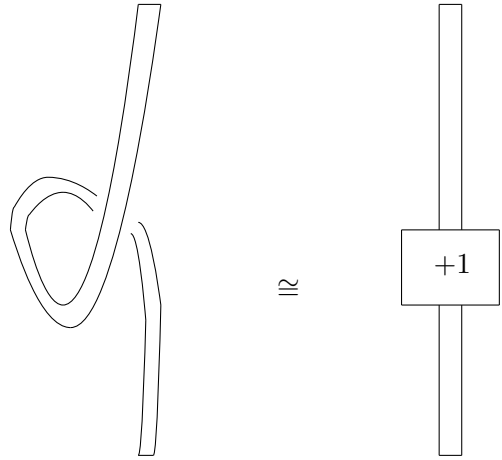
\begin{figure}
	\begin{center}
		\begin{tikzpicture}
		\draw (1,0) parabola (2,4);
		\draw (1.1,-0.3) parabola (2.3,4);	
		\draw (1,0) parabola (0.5, 1);
		\draw (0.55,1.2) parabola (0.5,1);
		\draw (1,1.5) parabola (0.55,1.2);
		\draw (1,1.5) parabola (1.4,1.25);
		\draw (1.1,-0.3) parabola (0.3, 1);
		\draw (0.35,1.3) parabola (0.3,1);
		\draw (0.8,1.7) parabola (0.35,1.3);
		\draw  (0.8,1.7) parabola (1.45,1.45);
		\draw (1.9,0.95) parabola (2.1,-0.2);
		\draw (2,-2) parabola (2.1,-0.2);
		\draw (2,1.1) parabola (2.3,0);
		\draw (2.2,-2) parabola (2.3,0);
		\node (a) at (3:4) {$\cong$};
		
		\draw (6,4) -- (6,1);
		\draw (6.3,4) -- (6.3,1);
		
		
		\draw (5.5,1) -- (6.8,1);
		\draw (5.5,1) -- (5.5,0);
		\draw (6.8,1) -- (6.8,0);
		\draw (5.5,0) -- (6.8,0);
		
		
		\draw (6,0) -- (6,-2);
		\draw (6.3,0) -- (6.3,-2);
		
		\node (a)  at (5:6.2) {$+1$};
		
		\draw (2,4)--(2.3,4);
		\draw (2,-2)--(2.2,-2);
		
		\draw (6,-2)--(6.3,-2);
		\draw (6,4)--(6.3,4);;
		\end{tikzpicture}
	\end{center}
	\caption{Self-crossing of a ribbon introduces twists.}
		\label{fig:selfcrossing}
\end{figure} 

Reidemeister moves (R moves for short) of type II and III translate directly into analogously defined moves where each arc is thickened to a ribbon, while R move I does not hold in the context of framed links, since it introduces a twist, i.e. a change in the framing. This is depicted in Figure~\ref{fig:selfcrossing}. Throughout this article we will depict positive, resp. negative, twists by a rectangle inserted in a ribbon with a positive, resp. negative, integer indicating the number of twists and their orientations. Isotopy equivalence of framed links is characterized by moves RII, RIII and the cancellation of twists depicted in Figure~\ref{fig:twistcancellation}, where each twist is thought of as a self-crossing as in Figure~\ref{fig:selfcrossing} (with negative twists obtained by kinks in the opposite direction).

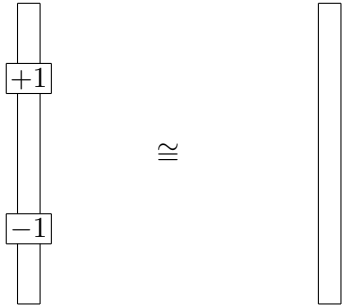
\begin{figure}
	\begin{center}
		\begin{tikzpicture}
		\draw (0,4) -- (0,3.2);
		\draw (0,2.8) -- (0,1.2);
		\draw (0,0.8) -- (0,0);
		
		\draw (0.3,4) -- (0.3,3.2);
		\draw (0.3,2.8) -- (0.3,1.2);
		\draw (0.3,0.8) -- (0.3,0);
		
		\draw (-0.15,3.2) -- (0.45,3.2);
		\draw (-0.15,2.8) -- (0.45,2.8);
		\draw (-0.15,1.2) -- (0.45,1.2); 
		\draw (-0.15,0.8) -- (0.45,0.8);
		\draw (-0.15,0.8) -- (-0.15,1.2);
		\draw (-0.15,2.8) -- (-0.15,3.2);
		\draw (0.45,0.8) -- (0.45,1.2);
		\draw (0.45,2.8) -- (0.45,3.2);
		
		\node (a) at (0.15,1) {$-1$};
		\node (a) at (0.15,3) {$+1$};
		
		\node (a) at (2,2) {$\cong$};
		
		\draw (4,4) -- (4,0);
		\draw (4.3,4) -- (4.3,0);
		
		\draw (4,4)--(4.3,4);
		\draw (4,0)--(4.3,0);
		
		\draw (0,4)--(0.3,4);;
		\draw (0,0)--(0.3,0);
		\end{tikzpicture}
	\end{center}
	\caption{Twists with opposite signs annihilate each other.}
	\label{fig:twistcancellation}
\end{figure}

\subsection{Ternary racks and their self-distributive cohomology}

Ternary racks are generalizations of racks to sets with ternary operations. Specifically, a set $X$ endowed with a ternary operation $T: X\times X\times X \longrightarrow X$ satisfying the conditions 
$$
T(T(x,y,z),u,v) = T(T(x,u,v),T(y,u,v),T(z,u,v)), 
$$
is said to be a ternary self-distributive (TSD) set. A TSD set such that the map $X \longrightarrow X$ defined by $T(\bullet, x,y)$ is a bijection for all $x,y\in X$ is said to be a ternary rack. A notable example of TSD structure is the heap of a group, defined as $(x,y,z) \mapsto xy^{-1}z$. Heap operations have been considered and studied in \cite{SZ}, in relation with their ribbon cocycle invariants (see Section~\ref{sec:ribboncocy} below). TSD operations naturally arise also by composing binary self-distributive operations. For instance, if $(Q,*)$ is a rack, or quandle, then the operation $T(x,y,z) := (x*y)*z$ can be seen to endow $Q$ with a TSD structure \cite{ESZ}.

We recall the notion of TSD (co)homology of ternary racks \cite{Green} and, more specifically, the TSD $2$-cocycle condition, since the ribbon cocycle invariant utilizes $2$-cocycles as weights, in a fashion similar to the original cocycle invariant introduced in \cite{CJKLS}. Let $(X,T)$ be a TSD set, and define $C_n(X)$ to be the free abelian group generated by $(2n+1)$-tuples of elements of $X$. Define maps $\partial_n : C_n(X) \longrightarrow C_{n-1}(X)$ by 
\begin{eqnarray*}
	\lefteqn{\partial_n (x_1, \ldots , x_{2n+1})} \\
	&& \sum_{i=1}^{n} (-1)^i[(x_1,\ldots, x_{2i-1}, \widehat{x_{2i}, x_{2i+1}},x_{2i+2}, \ldots , x_{2n+1})\\
	&& - (T(x_1,x_{2i},x_{2i+1}), \ldots , T(x_{2i-1},x_{2i},x_{2i+1}), \widehat{x_{2i} x_{2i+1}},x_{2i+2}, \ldots, x_{2n+1})],
	\end{eqnarray*}
and extended by $\Z$-linearity. A (long) direct computation shows that the maps $\partial_n$ are obtained as the alternating sum $\partial_n = \sum_{i=1}^n(-1)^i\partial_n^i$, where the maps $\partial_n^i$ satisfy the usual pre-simplicial module axioms and, consequently, $(C_n(X),\partial_n)$ defines a chain complex whose associated homology, written $H_n(X)$, is called TSD homology. Given an abelian group $A$ we obtain, upon dualizing the previous chain complex, TSD cochain groups and associated cohomology. We indicate these groups with the symbols $C^n(X;A)$ and $H^n(X;A)$, respectively. The $2$-cocycle condition, for a $2$-cochain $\psi: X^3\longrightarrow A$, takes the form 
\begin{eqnarray*}
\lefteqn{\psi(x,y,z) - \psi(T(x,u,v),T(y,u,v),T(z,u,v))}\\ 
 &=& \psi(x,u,v) - \psi(T(x,y,z),u,v) 
\end{eqnarray*}
for all $x,y,z,u,v\in X$. As it will be seen in Section~\ref{sec:ribboncocy}, the ternary $2$-cocycle condition, along with an appropriate interpretation of colorings of blackboard framings by ternary racks, is invariant under moves RII, RIII and cancellation move. It is therefore possible to define Boltzmann weights by means of ternary $2$-cocycles and introduce a state-sum invariant of framed links. Given a $1$-cochain $f:X\longrightarrow A$, the first cohomology differential $\delta^1$ maps it to the function 
$$
(x,y,z) \mapsto \delta^1 f(x,y,z) := f(x) - f(T(x,y,z)).
$$
Therefore two $2$-cocycles $\psi$ and $\phi$ are in the same second cohomology class if they differ by a term $\delta^1 f$ as above, for some $1$-cochain $f$. As it will be proved in Section~\ref{sec:ribboncocy}, changing the representative of a second cohomology class changes the ribbon cocycle invariant by a well understood term, so that the invariant is a well defined function, up to a known equivalence relation, of the cohomology group $H^2(X;A)$. This observation did not appear in the original construction in \cite{EZ}, and has been proven when $T$ is the heap operation in \cite{SZ}. 

\subsection{Braided monoidal categories and ribbon categories}

Recall that given a monoidal category $(\mathcal C, \otimes)$, a {\it braiding} in $\mathcal C$ is a natural family of isomorphisms $c_{X,Y}: X\otimes Y \longrightarrow Y\otimes X$ such that the {\it Hexagon Axiom} is satisfied \cite{Kas}, Chapter XIII. Specifically, it is required that the diagram 
\begin{center}
	\begin{tikzcd}
		&X\otimes (Y\otimes Z)\arrow[r,"c_{X,Y\otimes Z}"] &  (Y\otimes Z)\otimes X\arrow[rd,"\alpha_{Y,Z,X}"]& \\
		(X\otimes Y)\otimes Z\arrow[ur,"\alpha_{X,Y,Z}"]\arrow[dr,"c_{X,Y}\otimes \mathbbm 1"]& &  & Y\otimes (Z\otimes X)\\
		& (Y\otimes X)\otimes Z\arrow[r,"\alpha_{Y,X,Z}"] & Y\otimes (X\otimes Z)\arrow[ru,"\mathbbm 1\otimes c_{X,Z}"]&
		\end{tikzcd}
\end{center}
where $\alpha$ indicates the associativity constraint of the category $\mathcal C$, and we have omitted the subscript of the identity morphism, as no confusion arises. A similar diagram for the inverse of $c_{X,Y}$ is required to commute, but this can be obtained from the commutativity of the previous one. Therefore, it is not an independent axiom, see comment in \cite{Kas} right above Definition~XIII.1.1. A monoidal category endowed with a braiding is said to be a {\it braided monoidal category}. Observe that for any object $X\in \mathcal C$, the braiding $c_{X,X}$ is a solution to the braid (Yang-Baxter) equation. In fact, the diagrammatic interpretation coincides with the RIII move for knot/link diagrams. 

In this article we will consider our monoidal categories to be strict, so that the associativity constraints will not be written now on. This assumption is not particularly restrictive, as any monoidal category can be seen to be equivalent to a struct monoidal category. 

Typical examples of braided monoidal categories arise from braided bialgebras (see Chapter~VIII in \cite{Kas}), where the category of $H$-modules, of a braided bialgebra $H$, has a braided structure associated to the universal $R$-matrix of $H$. Another important class of examples arises from crossed $G$-sets, where the braiding is obtained using the crossed action. Linearizing these structures produces bradings in some  subcategory of vector spaces. More generally, one can use a quandle operation, which generalizes the axioms of crossed $G$-set. 

 A left dual of $X$ in a braided monoidal category is an object $X^*$ such that there exist morphisms $coev: \mathbb I \longrightarrow X\otimes X^*$ and $ev: X^*\otimes X \longrightarrow \mathbb I$ such that the equalities
 $(\mathbbm 1\otimes ev)\circ (coev\otimes \mathbbm 1) = \mathbbm 1$ and $(coev \otimes \mathbbm 1)\circ (\mathbbm 1 \otimes ev) = \mathbbm 1$ hold. A similar definition for right duals can be made. A category such that left and right duals exist for all objects is said to be {\it autonomous}, and in this case left and right duals coincide. In what follows, we will refer to left duality simply as ``duality'', if not otherwise specified. 
 
 In a braided monoidal category, the notion of dual introduces a diagrammatic interpretation with different types of crossing orientations. The corresponding RIII moves with new orientations are seen to be induced by the original diagrammatic RIII, as in Figure~10 in \cite{Tur}. See \cite{Kas}, Chapter~XIV, for the diagrammatic interpretation of duality. 

Given a braided monoidal category with duals, a {\it twist} is a natural family of isomorphisms $\theta_X$ such that $\theta_{X\otimes Y} = (\theta_X\otimes \theta_Y)\circ c_{Y,X}\circ c_{X,Y}$. Moreover, $\theta$ is required to behave well with respect to dual objects, in the sense that $\theta_{X^*} = (\theta_X)^*$. 

	In the previous example of braided monoidal categories arising from linearization of crossed $G$-sets, linear duals, and evaluation and coevaluation maps in vector spaces give a ribbon cateogry structure along with trivial twists. In fact, triviality of twists can be interpreted as a consequence of the fact that quandles are idempotent following Figure~\ref{fig:selfcrossing}. A twist is introduced by means of a self-crossing, and the corresponding effect of applying a quandle operation is trivial, due to idempotence. As it will be seen below, using ternary operations and their diagrammatic interpretation gives rise to ribbon categories, following a simialar paradigm, whose twisting morphisms are nontrivial. Consequently the corresponding invariants detect the framing of knot/link.  

	\section{Ribbon cocycle invariants}\label{sec:ribboncocy}
	Following \cite{CJKLS,CEGS}, it is introduced in this section an invariant of framed links, using colorings of ribbon diagrams by ternary quandles, and ternary quandle $2$-cocycles. This invariant was originally introduced in \cite{EZ}, and studied in \cite{SZ} in the case of heap invariants. We give the details of the construction, as they are relevant for the rest of the article. 
	
	Framed links are represented in the rest of the paper by their {\it blackboard framing}. Therefore the arcs of a projection on the plane are represented by ribbons bounded by two parallel arcs. Orientetions of the ribbons are specified by orientations of the parallel arcs, which will be always assumed to be concordant. The framing of a ribbon, which is an integer number, is obtained by twisting the two arcs delimiting the ribbon. This is given by consecutive self-intersections, and therefore a specified orientation of the ribbon determines whether $n$ consecutive twists are positive or negative. A diagram whose edges are specified by two parallel arcs, therefore defining a ribbon, is called {\it ribbon diagram}. It follows from the definitions that the blackboard framing of a framed link is a ribbon diagram. 
	
	Let $X$ be a tarnary quandle and let $\mathcal{D}$ be a diagram of a framed link. Suppose for the moment that the link has a single component, in other words it is a diagram of a framed knot. To each ribbon arc in $\mathcal{D}$, associate a color by a pair of elements $(x_1,x_2) \in X\times X$, corresponding to each side of the ribbon. At a positive crossing $\tau$ of $\mathcal{D}$, where the arcs colored by $(x_1^{\tau},x_2^{\tau})$ and $(y_1^{\tau},y_2^{\tau})$ meet, let the overpassing ribbon mantain the same color, while change the color of the underpassing ribbon to $(T(x_1,y_1,y_2),T(x_2,y_1,y_2))$. When a crossing $\tau$ is negative, we change the color of the underpassing ribbon to $(z_1,z_2)$, where $z_i$ is the unique element of $X$ such that $T(y_i,x_1,x_2) = z_i$, whose existence is guaranteed by the axioms of ternary quandle. We now pose the following.
	
	\begin{definition}\label{def:colorings}
		{\rm
		Let $\mathcal{D}$ be a ribbon diagram whose set of ribbon arcs is denoted by $\mathcal R$, and let $X$ be a ternary quandle. Then, a coloring of $\mathcal D$ by $X$, is a (set-theoretic) map 
		$$\mathcal C: \mathcal R \longrightarrow X\times X,$$ that is consistent with the coloring rule above.
		The set of colorings of a framed link is defined to be the set of colorings of a ribbon diagram of the link.
	}
	\end{definition}

\begin{lemma}\label{lem:coloring}
	Let $X$ be a ternary quandle and $\mathcal D$ a ribbon diagram of a framed link. Then the set of colorings $\mathcal C$ of $\mathcal D$ by $X$ is invariant under Reidemeister moves II and III, and moreover it respects cancelling of kinks.
\end{lemma}
As a consequence, the notion of coloring of a framed link is well posed, and the set of colorings of a framed link is an isotopy invariant. 

Suppose $\phi$ is a ternary quandle $2$-cocycle of $X$, with coefficients in $A$. For a given crossing $\tau$, define the Boltzmann weight at $\tau$, depending on the coloring $\mathcal{C}$ and the $2$-cocycle $\phi$ by $ (\phi (x^{\tau}_1,y^{\tau}_1,y^{\tau}_2)^{\epsilon(\tau)},\phi(x^{\tau}_2,y^{\tau}_1,y^{\tau}_2)^{\epsilon(\tau)})\in A\times A$, where the sign $\epsilon(\tau)$ is that of the crossing. Now we can define the $2$-cocycle invariant of a link $L$ with respect to a ternary quandle $X$ and a ternary $2$-cocycle $\psi\in Z^2(X,A)$ as follows. First we give the definition in the case of framed knots, and then generalize it to framed links. 

\begin{definition}\label{def:inv}
	{\rm
Let $\mathcal{D}$ be a ribbon diagram of a framed knot $K$ having $k$ crossings $\tau_1,\cdots ,\tau_k$, and let $\psi\in Z^2(X,A)$ be a ternary $2$-cocycle. Define the cocycle invariant of $K$, by the $2$-cocycle $\psi$ as
$$
\Theta_{\mathcal D} (X,T,\psi) =\sum_{\mathcal C}   (\prod_{i=1}^{k}\psi (x^{\tau_i}_1,y^{\tau_i}_1,y^{\tau_i}_2)^{\epsilon(\tau_i)},\prod_{i=1}^{k}\psi(x^{\tau_i}_2,y^{\tau_i}_1,y^{\tau_i}_2)^{\epsilon(\tau_i)}),
$$	
where $\epsilon(\tau_i)$ is the sign of the $i^{th}$ crossing, and the sum runs over all colorings $\mathcal C$ of the diagram $D$.  We sometimes shorten notation by writing $\Theta_\psi$ instead of $\Theta_{\mathcal D} (X,T,\psi)$, and name $\Theta_{\mathcal D} (X,T,\psi)$ the ribbon cocycle invariant. Each term $\psi (x^{\tau_i}_j,y^{\tau_i}_1,y^{\tau_i}_2)^{\epsilon(\tau_i)}$, $j = 1,2$, is also called the {\it Boltzmann weight} at $\tau$, relative to the coloring $\mathcal C$ and the $2$-cocycle $\psi$. It is denoted by the symbol $\mathcal B_j(\mathcal C,\tau)$.
}
\end{definition}

It remains to prove that the $2$-cocycle invariant does not depend on the choice of ribbon diagram  $\mathcal D$ used in Definition~\ref{def:inv}.  
\begin{theorem}\label{th:inv}
	The cocycle invariant does not depend on the equivalence class of the ribbon diagram $\mathcal{D}$. Therefore, it is well defined and it is an invariant of framed knots. Moreover, changing a $2$-cocycle $\psi$ to another representative of the same second cohomology group $[\psi]\in H^2(X,A)$, changes the invariant to an integer multiple of the unit $(e,e) \in A\times A$.
\end{theorem}
\begin{proof}
The Theorem is proved by showing that the state sum is invariant with respect to Reidemeister moves II and III, and with respect to kink cancellation. Compare with the moves T1-T5 and T6$_f$ given in \cite{FY}, page 166, and with rel$_1$-rel$_8$ in \cite{RT}, page 14. As seen in Lemma~\ref{lem:coloring}, applying a Reidemeister move II or III, or applying the kink cancellation move transforms one coloring of a diagram to another coloring. Assuming that the colors at the top strands, hen changing colorings between Reidemeister move II, are $(x,y)\times (z,w)$, we see that the consecutive crossings $\tau_1$ and $\tau_2$ contribute with a cocycle weight of $(\psi(x,z,w),\psi(y,z,w)$ and $(\psi(x,z,w)^{-1},\psi(y,z,w)^{-1})$ respectively. So the contributions cancel. Reidemeister mover III coincides with with the ternary $2$-cocycle condition, and the corresponding invariance is guaranteed. Cancellation of kinks is similar to the case of Reidemeister move II, and the sign in Definition~\ref{def:inv} ensures that the state sum does not change. Suppose $\psi$ is a coboundary, i.e. $\psi = \delta f$, for some $f:X\longrightarrow A$. Let $\mathcal C$ be a fixed coloring and suppose we order the crossings $\tau_1, \ldots , \tau_n$ starting from an arbitrarily chosen arc, and numbering the crossings as we encounter them following the arc along the knot. Consider $\tau_i$ and $\tau_{i+1}$, and assume without loss of generality that $\tau_i$ is a positive crossing. Then the Boltzmann sum contribution on the first entry of the state sum is given by $f(x)f(T(x,z,w))^{-1}$, assuming that the colorings on top of $\tau_i$ are $(x,y)\times (z,w)$. Likewise, the contribution on the second entry is $f(y)f(T(y,z,w))^{-1}$. Now, suppose $\tau_{i+1}$ is positive, and let $(z',w')$ overpass at $\tau_{i+1}$. Then the Boltzmann sum is given by $f(T(x,z,w))f(T(T(x,z,w),z',w'))^{-1}$, which shows that the two terms $f(T(x,z,w))$ and $f(T(x,z,w))^{-1}$ cancel. Similarly for the second entry of the Boltzmann weight. If $\tau_{i+1}$ is negative, and we let $(z',w')$ be underpassing, we see that the terms $f(T(x,z,w))$ and $f(T(y,z,w))$ again appear in both crossings with opposite signs, and cancel out again. Proceeding in this fashion we see that when we have $\tau_n$ and $\tau_1$ the remaining terms cancel out, since have assumed that $\mathcal C$ is a colorng, hence it is well defined. Each coloring in the state sum contributes with a trivial term, so the invariant of a coboundary simply counts the number of colorings. Consequently, it is an integer multiple of $(e,e)$. This fact implies that if $\psi$ and $\phi$ differ by a coboundary, their invariants differ by an integer multiple of $(e,e)$. A more concise argument, based on an interpretation of the state-sum invariant in terms of Kronecker pairing first described in \cite{CJKLS}, is given in Theorem~5.8 of \cite{SZ} for the case when $T$ is the heap of a group. The reader can verify that it is applicable also for general TSD operations.
	\end{proof}
\begin{remark}
		{\rm 
	The invariant $\Theta_{\mathcal D}(X,T,\psi)$ is, by construction, an element of the group ring $\Z [A\times A]$. Since there is an isomorphism $\Z[A\times A] \cong \Z[A]\otimes \Z[A]$, we can identify $\Theta_{\mathcal D}(X,T,\psi)$ with a sum of tensor products of elements of $A$. 
}
\end{remark}

We now generalize Definition~\ref{def:inv} to framed links with multiple components. First, if a link $\mathcal L$ has $t$ components $\mathcal L = \mathcal K_1 \cup \cdots \cup \mathcal K_t$, we label the crossings $\tau$ of $\mathcal L$ with the number of the component the underpassing ribbon belongs to. Therefore, for example, if at the crossing $\tau$ the underpassing ribbon belongs to the component $i$, $\tau$ is denoted by $\tau_i$. For $j = 1,2$ we define the Boltzmann weight $B_j^{(i)}(\mathcal C,\tau_i)$ relative to the crossing $\tau_i$, in the $i^{\rm th}$ component, as $\psi(x_j, y_1,y_2)^{\epsilon (\tau_i)}$, where $(y_1,y_2)$ is the coloring of the overpassing ribbon (not necessarily in the component $K_i$) and $(x_1,x_2)$ is the coloring of the underpassing ribbon (in the component $\mathcal K_i$ by assumption). In the following definition, we denote a vector with multiple entries being pairs with the notation $(a,b)\times \cdots \times (a',b')$. 

\begin{definition}
	{\rm 
Let the notation be as in the previous paragraph. Then the (vector) ribbon cocycle invariant of $\mathcal L$, relative to the ternary quandle $X$ and the ternary $2$-cocycle $\psi$ is defined as
$$
\vec{\Theta}_{\mathcal D} (X,T,\psi) =
 \sum_{\mathcal C} \times_{i=1}^t(\prod_{s=1}^{k(i)}\mathcal B_1^i(\mathcal C, \tau_i(s)),\prod_{s=1}^{k(i)}B_2^i(\mathcal C, \tau_i(s))),
$$
where $k(i)$ is the number of crossings in the $i^{\rm th}$ component, $\tau_i(s)$ indicates the $s^{\rm th}$ crossing in the $i^{\rm th}$ component, and the sum  indicates that in each component of the vector we are summing over all possible colorings $\mathcal C$. 
}
\end{definition}

An argument similar to that of Theorem~\ref{th:inv}, applied to each component, shows that the (vector) ribbon cocycle invariant does not depend on the isotopy class of the framed link $\mathcal L$ and it is therefore well defined.  Moreover, changing $\psi$ to another $2$-cocycle in the same second cohomology class changes the invariant by an integer multiple of the vector $(e,e)\times \cdots \times (e,e)$, where $e$ is the neutral element of the coefficient group used for cohomology. 
\section{Ribbon categories from self-distributive ternary operations}\label{sec:ribboncat}

In this section it is given a generalization of known constructions that allow to define a Yetter-Drinfeld module from a set-theoretic quandle \cite{Gra}. More specifically, the aim of this section is to define a ribbon category given a ternary self-distributive object in the symmetric monoidal category of vector spaces. As opposed to the case of a quandle, the ribbon category that is obtained in this procedure admits nontrivial twisting morphisms. These are defined by means of a self-crossing, similar to the Reidemeister move I, which in this case is not the trivial map. In this section, along with sections~\ref{sec:ribbonquantum} and~\ref{sec:examples} we focus on the case of ternary self-distributive objects obtained via linearization of set-theoretic structures. We use ternary cohomology in the usual sense~\cite{ESZ}. It will be shown in Section~\ref{sec:generalized} that this construction can be generalized to the case of ternary self-distributive objects in symmetric monoidal categories endowed with duals, using a generalized version of $2$-cocycle condition with coefficients in a group object of the given category. 


Let $Q := \{x_1,\ldots ,x_n\}$ be a finite ternary quandle with operation $T: Q\times Q\times Q\longrightarrow Q$. Then linearizing the operation over a ground field $\mathbbm k$ gives a ternary self-distributive object in the category of $\mathbbm k$-vector spaces, $(X, T)$, where $X:= \mathbbm k\langle x_1,\ldots , x_n\rangle$ and the linearized operation is indicated with the same symbol $T$.  

Let $H^2(Q,A)$ indicate the second ternary self-distributive cohomology group of $Q$, with coefficients in the multiplicative (abelian) group $A$. Fix a nontrivial 2-cocycle $\alpha: Q\times Q\times Q \longrightarrow A$. In the examples treated below we have the case when $A \subset \mathbbm k^{\times}$ or, more generally, when it is given a group character of $A$ in the group of units of $\mathbbm k$. The group $A$ therefore acts on the vector space $X$ and its tensor products via scalar multiplication of $\mathbbm k$.  We will use the symbol$\cdot$ to indicate this action, for clarity, and we use juxtaposition to denote the multiplication in $A$. 

Using this data, we define the brading $c^{\alpha}: X^{\otimes 2} \otimes X^{\otimes 2}\longrightarrow X^{\otimes 2} \otimes X^{\otimes 2}$ by the assignment
$$
(x\otimes y) \otimes (z\otimes w) \mapsto \alpha(x,z,w)\alpha(y,z,w)\cdot (z\otimes w)\otimes (T(x,z,w)\otimes T(y,z,w)), 
$$
having used a comma to separate the entries of $T$ instead of the symbol $\otimes$ to shorten notation. As it is proved below, the morphism $c^{\alpha}$ satisfies the braid equation if and only if $\alpha$ is a 2-cocycle, i.e. $[\alpha]\in H^2(Q, A)$ (cf. with Proposition 3.3 in \cite{Gra}). The twisting morpshism $\theta^{\alpha}_2 : X^{\otimes 2} \longrightarrow X^{\otimes 2}$ is defined by extending the assignement
$$
x\otimes y\mapsto \alpha(x,x,y)\alpha(y,x,y)\cdot  T(x,x,y)\otimes T(y,x,y).
$$

 This definition is motivated by introducing a complete twist in a ribbon by self crossing, Figure~\ref{fig:selfcrossing}. 

We introduce now the ribbon category whose objects are all even tensor powers of the vector space $X$ and generated by braiding and twisting given above. See\cite{Kas} Chapter XII, Section XII.1, for the general definition of presentation of a category. We give a very explicit definition below, to describe the morphisms in detail. 
  
\begin{definition}\label{def:ribboncat}
	{\rm 
	Let $(X,T)$ be a ternary self-distributive object in $\mathcal{V}_{\mathbbm{k}}$, as above, and $[\alpha]\in H^2(Q,A)$. Define the category $\mathcal{R}_{\alpha}(X)$ as follows. The objects are even powers $X^{\otimes 2n}$ of $X$ in the category $\mathcal{V}_{\mathbbm{k}}$. The tensor product of two objects $Y$ and $Z$, written $Y\boxtimes Z$, is defined to be the tensor product $\otimes$ in $\mathcal{V}_{\mathbbm{k}}$. The trivial power $X^0$ is set to be $\mathbbm{k}$ by definition. The morphisms of this category are defined as follows. The set ${\rm Hom}(X^{\otimes 2},X^{\otimes 2})$ consists of the identity map and twists $(\theta^\alpha)^{\circ m}$, where $m\in \Z$ and $\circ$ indicates composition. The set ${\rm Hom}(Y,Z)$ is empty if $Y\neq Z$. The morphism set ${\rm Hom}(X^{\otimes 4},X^{\otimes 4})$ is the free monoid generated under composition by twofold tensor products of $f,g\in {\rm Hom}(X^{\otimes 2},X^{\otimes 2})$, and the braiding $c^\alpha := c^\alpha_{2,2}$. The set ${\rm Hom}(X^{\otimes 2n},X^{\otimes 2n})$ is defined inductively as the free monoid generated under composition by tensor products $f\in {\rm Hom}(X^{\otimes 2m_1},X^{\otimes 2m_1})$ and $g\in {\rm Hom}(X^{\otimes 2m_2},X^{\otimes 2m_2})$ with $m_1 + m_2 = n$. The braidings $c^\alpha_{n,m} : X^{\otimes 2n}\boxtimes X^{\otimes 2m} \longrightarrow X^{\otimes 2m}\boxtimes X^{\otimes 2n}$ are the morphisms $X^{2(n+m)} \longrightarrow X^{2(n+m)}$ corresponding to block permutation switching $X^{\otimes 2n}$ and $X^{\otimes 2m}$, obtained by subsequent applications of $c^\alpha_{2,2}$. 
}
\end{definition}

Endow the category $\mathcal{R}_{\alpha}(X)$ with duals by setting $(X^{\otimes 2n})^* := (X^*)^{\otimes 2n}$, where $X^*$ is the linear dual of $X$. The evaluation map $ev$ is determined by $x\otimes y\boxtimes f\otimes g \mapsto f(y)g(x)$, and the coevaluation map $coev$ by $1\mapsto x_j\otimes x_i \boxtimes x^i\otimes x^j$, where the Einstein summation convention is used. Let $\mathcal{R}^*_{\alpha}(X)$ indicate the category $\mathcal{R}_{\alpha}(X)$ with the addition of duals. 

\begin{remark}\label{rmk:twist}
	{\rm 
 Direct inspection shows that the twisting morphism $\theta^\alpha$ defined above coincides with the self-intersection morphism obtained by means of $ev$ and $coev$ maps as composition: $(ev\otimes \mathbbm 1^{\otimes 2})\circ (\mathbbm 1^{\otimes 2}\otimes c^\alpha)\circ (coev\otimes \mathbbm 1^{\otimes 2})$, cf. Figure~\ref{fig:selfcrossing}. 
}
\end{remark}

\begin{theorem}\label{thm:ribboncat}
	With the same notation as above, the category $\mathcal{R}^*_{\alpha}(X)$ with braiding induced by $c^\alpha$ and twisting morphisms induced by $\theta^\alpha$ is a ribbon category. Moreover, if $[\alpha] = [\beta]$, then the two categories $\mathcal{R}^*_{\alpha}(X)$ and $\mathcal{R}^*_{\beta}(X)$ are equivalent. 
\end{theorem}
\begin{proof}
	The fact that $\mathcal{R}_{\alpha}(X)$ is a tensor category is a consequence of the fact that $\mathcal V_{\mathbbm{k}}$ is a tensor category, and the definitions. To verify naturality of the braiding, it is enough to verify the commutativity of the square
	$$
	\begin{tikzcd}
	X^{\otimes 2n}\boxtimes X^{\otimes 2m}\arrow[r,"c^\alpha_{2n,2m}"]\arrow[d,"f\boxtimes g"] & X^{\otimes 2m}\boxtimes X^{\otimes 2n}\arrow[d,"g\boxtimes f"]\\
	X^{\otimes 2n}\boxtimes X^{\otimes 2m}\arrow[r,"c^\alpha_{2n,2m}"]& X^{\otimes 2m}\boxtimes X^{\otimes 2n}
	\end{tikzcd}
	$$
	for all morphisms $f$ and $g$. This follows from the fact that the morphisms are defined as block braidings induced by $c^\alpha$. So the naturality is a direct consequence of the definition of morphisms, and the fact that $c^\alpha$ satisfies the braid equation (proved below) in the same way it is proved for the braid category. The fact that $c^\alpha$ is a family of isomorphisms is a consequence of the invertibility of $T$ axiom, holding for ternary racks. To verify that $c^\alpha$ is indeed a braiding, it has to be shown that it satisfies the braid equation. To this objective, observe first that it is enough to prove that $c^\alpha_{2,2}$ satisfies the braid equation
	$$
  (c^\alpha_{2,2}\boxtimes \mathbbm{1}) \circ ( \mathbbm{1}\boxtimes c^\alpha_{2,2})\circ (c^\alpha_{2,2}\boxtimes \mathbbm{1}) = ( \mathbbm{1}\boxtimes c^\alpha_{2,2})\circ (c^\alpha_{2,2}\boxtimes \mathbbm{1})\circ ( \mathbbm{1}\boxtimes c^\alpha_{2,2}),
    $$
	 since the general case would follow from iterations of this specific case. Using the shorthand notation $T_x^{y,z} = T(x,y,z)$, on a general basis vector $x\otimes y\boxtimes z\otimes w\boxtimes u\otimes v$, the left hand side equals	
	 \begin{eqnarray*}
	 	\lefteqn{(c^\alpha_{2,2}\boxtimes \mathbbm{1}) \circ ( \mathbbm{1}\boxtimes c^\alpha_{2,2})\circ (c^\alpha_{2,2}\boxtimes \mathbbm{1})x\otimes y\boxtimes z\otimes w\boxtimes u\otimes v }\\
	 	&=& \alpha(x,z,w)\alpha(y,z,w)\alpha(T_x^{z,w},u,v)\alpha(T_y^{z,w},u,v)\alpha(z,u,v)\alpha(w,u,v)\\
	 	&& \cdot u\otimes v\boxtimes T_z^{u,v}\otimes T_w^{u,v}\boxtimes T_{T_x^{z,w}}^{u,v}\otimes T_{T_y^{z,w}}^{u,v},
	 	\end{eqnarray*}
 	while the right hand side equals
 	\begin{eqnarray*}
 		\lefteqn{( \mathbbm{1}\boxtimes c^\alpha_{2,2})\circ (c^\alpha_{2,2}\boxtimes \mathbbm{1})\circ ( \mathbbm{1}\boxtimes c^\alpha_{2,2}) x\otimes y\boxtimes z\otimes w\boxtimes u\otimes v}\\
 		&=& \alpha(z,u,v)\alpha(w,u,v)\alpha(x,u,v)\alpha(y,u,v)\alpha(T_x^{uv},T_z^{uv},T_w^{uv})\alpha(T_y^{uv},T_z^{uv},T_w^{uv})\\
 		&& \cdot u\otimes v\boxtimes T_z^{u,v}\boxtimes T_w^{u,v}\boxtimes T_{T_x^{u,v}}^{T_z^{u,v},T_w^{u,v}}\otimes T_{T_y^{u,v}}^{T_z^{u,v},T_w^{u,v}}.
 		\end{eqnarray*}
 	The two terms are seen to coincide by applying the $2$-cocycle condition to $x,z,w,u,v$ and $y,z,w,u,v$ separately, and the definition of self-distributivity of $T$. The fact that the duals turn $\mathcal{R}_{\alpha}(X)$ into a rigid category is standard. It is left to prove that the twisting morphisms are naturural with respect to the braiding, i.e. that 
 	$$
 (	\theta^\alpha \boxtimes \theta^\alpha ) \circ c^\alpha = c^\alpha \circ ( \theta^\alpha \boxtimes \theta^\alpha ),
 	$$
 	and that 
 	$$
 	\theta^\alpha_{Y\boxtimes Z} = c^\alpha_{Z,Y}\circ c^\alpha_{Y,Z}\circ (\theta^\alpha_Y\boxtimes \theta^\alpha_Z).
 	$$
 	The latter follows immediately from the definition of $\theta^\alpha$ as extension of $\theta^\alpha_2$. To prove naturality, observe that it is enough to show that 
 	$$
 	(\theta^\alpha\boxtimes \mathbbm{1}) \circ c^\alpha_{2,2} = c^\alpha_{2,2} \circ (\mathbbm{1}\boxtimes \theta^\alpha) 
 	$$
 	and 
 	$$
 	(\mathbbm{1}\boxtimes \theta^\alpha) \circ c^\alpha_{2,2} = c^\alpha_{2,2} \circ (\theta^\alpha \boxtimes \mathbbm{1}),
 	$$
 	since the general naturality with respect of the braiding is obtained diagrammatically by sliding twists below and above a crossing, as in Figure~\ref{fig:twistcrossing}. The case of a twist being slid below a crossing is similarly depicted. 
 	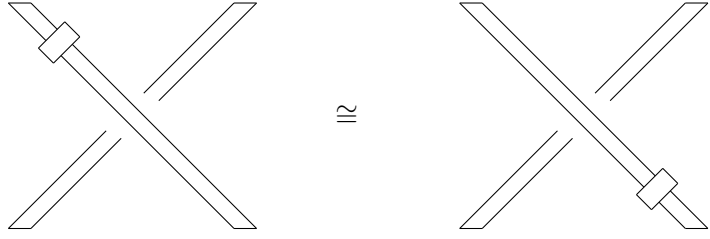
\begin{figure}
 		\begin{center}
 			\begin{tikzpicture}
 			\draw (0,3) -- (0.5,2.5);
 			\draw (0.3,3)--(0.65,2.65);
 			\draw (0.4,2.4)--(0.75,2.75);
 			\draw (0.4,2.4)--(0.6,2.2);
 			\draw (0.75,2.75)--(0.95,2.55);
 			\draw (0.6,2.2)--(0.95,2.55);
 			\draw (0.7,2.3)--(3,0);
 			\draw (0.85,2.45)--(3.3,0);

 			\draw (3,3) -- (1.8,1.8);
 			\draw (1.3,1.3) -- (0,0);
 			
 			\draw (3.3,3) -- (2,1.7);
 			\draw (1.5,1.2) -- (0.3,0);
 			
 			\draw (0,3) -- (0.3,3);
 			\draw (3,3) -- (3.3,3);
 			\draw (0,0) -- (0.3,0);
 			\draw (3,0) -- (3.3,0);
 			
 			
 			\node (a) at (4.5,1.5) {$\cong$};
 			
 			
 			\draw (6,3)--(6.3,3);
 			\draw (9,0)--(9.3,0);
 			
 			
 			\draw (8.9,0.6)--(8.55,0.25);
 			\draw (8.9,0.6)--(8.7,0.8);
 			\draw (8.55,0.25)--(8.35,0.45);
 			\draw (8.7,0.8)--(8.35,0.45);
 			
 			\draw (8.6,0.7)--(6.3,3);
 			\draw (8.45,0.55)--(6,3);
 			\draw (9,0)--(8.65,0.35);
 			\draw (9.3,0)--(8.8,0.5);

 			\draw (9,3) -- (7.8,1.8);
 			\draw (7.3,1.3) -- (6,0);
 			
 			\draw (9.3,3) -- (8,1.7);
 			\draw (7.5,1.2) -- (6.3,0);
 			\draw (6,0) -- (6.3,0);
 			\draw (9,3)--(9.3,3);
 			\end{tikzpicture}
 		\end{center}
 	\label{fig:twistcrossing}
 	\caption{A twist can be slid over a crossing.}
 	\end{figure} 
 	On simple tensors $x\otimes y\boxtimes z\otimes w$, the left hand side of the first equality becomes
 	\begin{eqnarray*}
 		\lefteqn{(\theta^\alpha \boxtimes \mathbbm{1}) \circ c^\alpha_{2,2} x\otimes y\boxtimes z\otimes w}\\
 		&=& \alpha(x,x,y)\alpha(y,x,y)\alpha(T_x^{z,w},z,w)\alpha (T_y^{z,w},z,w)\\
 		&& \cdot z\otimes w\boxtimes T_{T_x^{z,w}}^{z,w}\otimes T_{T_y^{z,w}}^{z,w},
 		\end{eqnarray*}
 	while the right hand side is
 	\begin{eqnarray*}
 	\lefteqn{c^\alpha_{2,2} \circ (\mathbbm{1}\boxtimes \theta^\alpha)x\otimes y\boxtimes z\otimes w }\\
 	&=& \alpha(x,z,w)\alpha(y,z,w) \alpha(T_x^{z,w},T_x^{z,w},T_y^{z,w})\alpha(T_y^{z,w},T_x^{z,w},T_y^{z,w})\\
 	&& \cdot T_z^{z,w}\otimes T_w^{z,w}\boxtimes T_{T_x^{z,w}}^{T_x^{z,w},T_y^{z,w}}\otimes T_{T_y^{z,w}}^{T_x^{z,w},T_y^{z,w}}.
 	\end{eqnarray*}
 These are seen to coincide upon applying the $2$-cocycle condition twice to $x,x,y,z,w$ and $y,x,y,z,w$, and using self-distributivity of $T$.  Similarly, for the second equality to be verified one has on the left hand side
 \begin{eqnarray*}
 \lefteqn{(\mathbbm{1}\boxtimes \theta^\alpha) \circ c^\alpha_{2,2}x\otimes y\boxtimes z\otimes w}\\
 &=& \alpha(z,z,w)\alpha(w,z,w)\alpha(x,T_z^{z,w},T_w^{z,w})\alpha(y,T_z^{z,w},T_w^{z,w})\\
 && \cdot T_z^{z,w}\otimes T_w^{z,w}\boxtimes T_x^{T_z^{z,w},T_w^{z,w}}\boxtimes T_x^{T_z^{z,w},T_w^{z,w}}.
 \end{eqnarray*}
To verify that they are the same, apply the $2$-cocycle condition to $T_x^{z,w},z,w,z,w$ and observe that 
\begin{eqnarray*}
\lefteqn{T(x,z,w)}\\
 &=& T(T(T(x,z,w),z,w),z,w)\\
 &=& T(T(T(x,z,w),z,w)T(z,z,w),T(w,z,w))\\
 &=& T(x,T(z,z,w),T(w,z,w)),
\end{eqnarray*}
and similarly $T(y,z,w) = T(y,T(z,z,w),T(w,z,w))$. This sequence of equalities is motivated diagrammatically by sliding a ribbon beneath a self crossing. 

Suppose now that $[\alpha] = [\beta]$, i.e. there exists $f : X\longrightarrow A$ such that $\alpha(x,y,z) = \beta(x,y,z) f(x)f(T(x,y,z))^{-1}$ for all $x,y,z$. Set $\tilde f: X^{\otimes 2} \longrightarrow X^{\otimes 2}$ as $\tilde f(x\otimes y) := f(x)f(y) \cdot x\otimes y$, and extended by linearity. The map $\tilde f$ has an inverse given by $\tilde f^{-1}(x\otimes y) := f(x)^{-1}f(y)^{-1} \cdot x\otimes y$. The definition of $\tilde f$ and its inverse clearly extends to objects of $\mathcal{R}^*_{\alpha}(X)$  and $\mathcal{R}^*_{\beta}(X)$. Define a functor $F: \mathcal{R}^*_{\beta}(X) \longrightarrow \mathcal{R}^*_{\alpha}(X)$ as follows. On objects $F$ is defined to be the identity. On morphisms $\tau\in {\rm Hom}(X^{\otimes 2n},X^{\otimes 2n})$ define $F(\tau) = (\tilde f^{-1})^{\otimes n}\circ \tau \circ (\tilde f^{\otimes n})$. This assignment is functorial, and moreover maps $\theta^\beta_2$ to $\theta^\alpha_2$ and $c^\beta_{2,2}$ to $c^\alpha_{2,2}$, since 
\begin{eqnarray*}
\lefteqn{(\tilde f^{-1} \theta^\beta_2 \tilde f)(x\otimes y)} \\
&=& (f(x)f(y) \beta(x,x,y)\beta(y,x,y)f(T(x,x,y))^{-1}f(T(y,x,y))^{-1}) \\
&& \cdot  x\otimes y\\
&=& \alpha(x,x,y)\alpha(y,x,y) \cdot x\otimes y\\
&=& \theta^\alpha_2 (x\otimes y)
\end{eqnarray*}
\begin{eqnarray*}
\lefteqn{(\tilde f^{-1} c^\beta_{2,2} \tilde f)(x\otimes y\boxtimes z\otimes w)} \\ &=&(f(x)f(y)\beta(x,z,w)\beta(y,z,w)f(T((x,z,w))^{-1}f(T((y,z,w))^{-1})\\
&& \cdot z\otimes w\boxtimes T((x,z,w)\otimes T((y,z,w)\\
&=& \alpha(x,z,w) \alpha(y,z,w)\cdot z\otimes w\boxtimes T((x,z,w)\otimes T((y,z,w)\\
&=& \theta^\alpha_{2,2} (x\otimes y\boxtimes z\otimes w).
\end{eqnarray*}
 The definition of $F$ clearly respects tensor products and due to the inductive definition of twists and braiding in the categories $\mathcal{R}^*_{\alpha}(X)$ and $\mathcal{R}^*_{\beta}(X)$, it follows that $F$ is a braided tensor functor. Since $F$ is essentially surjective on objects (it is the identity), and fully faithful on morphisms (due to invertibility of $\tilde f$), $F$ is an equivalence of braided tensor categories that respects the twisting structure. This completes the proof. 
\end{proof}

\begin{remark}
	{\rm 
Observe that the twisting morphisms $\theta^\alpha$ are in general nontrivial, even for a trivial $\alpha$.  	
}
\end{remark}

The main purpose of the construction of Theorem~\ref{thm:ribboncat}, is to show that the natural invariants associated to the ribbon category coincide with the cocycle invariants defined in Section~\ref{sec:ribboncocy}. It is in fact possible to bypass this construction and prove that the twisting morphsism $\theta^\alpha$ and the braiding morphism $c^\alpha_{2,2}$ induce a representation of the infinite framed braid group (the inductive limit of the framed braid groups $FB_n$ of \cite{KS}) similarly to \cite{Tur}, using the linear map $\Phi_b$ of Section~\ref{sec:ribbonquantum} (given below). It is convenient, though, to define a ribbon category from a ternary self-distributive structure and a ternary $2$-cocycle since this is suitable for a generalization to multiple compatible self-distributive structures.
\begin{definition}\label{def:compatiblestructures}
	{\rm
	A {\it compatible system of ternary self-distributive structures} is a finite family $\{(Q_i,T_i)\}_{i\in I}$ of ternary self-distributive sets along with actions
	$$
	T_{ij}: Q_i\times Q_j\times Q_j \longrightarrow Q_i ,
	$$
	of $Q_j\times Q_j$ on $Q_i$ for all $i,j =1,\ldots n$, satisfying the compatibility condition
	$$
	T_{ik}\circ (T_{ij}\times \mathbbm{1}\times \mathbbm{1}) = T_{ij}\circ (T_{ik}\times T_{jk}\times T_{jk})\circ \shuffle \circ (\mathbbm{1}^{\times 3}\times \Delta_3\times \Delta_3),
	$$
	where $\Delta_3 = (\Delta\times \mathbbm{1})\circ \Delta$ and $\shuffle$ is the permutation map corresponding to ternary self-distributivity. Such a system is denoted $\{Q_i,T_{ij}\}$, where it is implicitly assumed that $T_{ii} := T_i$.
}
\end{definition} 

Linearizing such a system of compatible structues gives a multi-object analogue of ternary self-distributive object in the category of vector spaces. We will sometime call the linearized object a compatible system of self-distributive structures, since no confusion will arise. 

\begin{remark}\label{rmk:mutuallydist}
	{\rm
There are infinitely many examples of the structure given in Definition~\ref{def:compatiblestructures} arising from mutually distributive structures as in \cite{ESZ}. In fact let $(X,T_0,T_1)$ be a mutually self-distributive set (with $X$ finite). Linearizing $T_0$ and $T_1$ over a field $\mathbbm{k}$ and defining $T_{10} = T_0 = T_{00}$ and $T_{01} = T_1 = T_{11}$ over the vector space, it is seen by direct inspection that the structure $\{\mathbbm{k}\langle X\rangle ,T_{ij}\}_{i,j = 0,1}$ is a compatible system of ternary self-distributive structures. 
}
\end{remark}
In Section~\ref{sec:examples} it will be seen that compositions of $G$-families of quandles provide other natural examples of these structures. Furthermore, in the Appendix, more examples from augmented Hopf modules will be introduced. In particular it will be seen that there are compatible systems with multiple base spaces. 

To the notion of compatible system of ternary self-distributive structures, there corresponds the notion of compatible system of ternary $2$-cocycles as follows. 

\begin{definition}\label{def:compatible2cocy}
	{\rm 
		Let $\{(Q_i,T_{ij})\}$ be a compatible system of ternary self-distributive structures.  A {\it compatible system of $2$-cocycles} with coefficients in $A$ (abelian group with multiplicative notation) is a family of maps $\alpha_{ij}: X_i\times X_j\times X_j \longrightarrow A$ such that
		\begin{eqnarray*}
		\lefteqn{\alpha_{ij}(x, y, z) \cdot \alpha_{ik}(T_{ij}(x, y, z), u, v)}\\
			&=& \alpha_{ik}(x, u, v) \cdot \alpha_{ij}(T_{ik}(x, u, v),T_{jk}(y, u, v), T_{jk}(z, u, v)),
		\end{eqnarray*}
	for all $x\in X_i$, $y,z\in X_j$ and $u,v\in X_k$ and all $i,j,k$. Such a family of maps is denoted by the symbol $\{\alpha_{ij}\}$, where parentheses can be omitted to shorten notation. 
}
\end{definition}

\begin{definition}\label{def:trivial2cocy}
	{\rm 
A compatible system of $2$-cocycles $\{\alpha_{ij}\}$ is said to be trivial, or cobounded, if there exists a family of maps $f_i : X_i \longrightarrow A$ such that 
$$
\delta f_i (x\times y_1\times y_2) := f_i(x) f_i(T_{ij}(x\times y_1\times y_2))^{-1} = \alpha_{ij} (x\times y_1\times y_2)
$$
for all $i,j$ and all $x\in X_i$, $y_1,y_2\in X_j$. 

Two systems of $2$-cocycles, $\{\alpha_{ij}\}$ and $\{\beta_{ij}\}$ are said to be equivalent if the system $\{\alpha_{ij}\beta^{-1}_{i,j}\}$ is trivial. 
}
\end{definition}
\begin{remark}\label{rmk:nontrivialcocy}
	{\rm 
	Observe that when $i=j=k$ it follows that each $\alpha_{ii}$ is a ternary $2$-cocycle for the ternary self-distributive structure $T_i$, and moreover the triviality condition gives that $\alpha_{ii}$ represents the trivial class in second ternary self-distributive cohomology group. It follows that if $\alpha_{ii}$ is not cobounded in the ternary self-distributive cochain, then a compatible system of $2$-cocycles is nontrivial in the sense of Definition~\ref{def:trivial2cocy}.
}
\end{remark}

\begin{remark}
	{\rm 
	Definitions~\ref{def:compatible2cocy} and~\ref{def:trivial2cocy} are clearly reminiscent of a cohomology theory. It is natural to ask whether such a theory derives from a deformation theory of compatible systems. The answer is no, in that infinitesimal deformations of compatible systems require more conditions to be satisfied, than just the $2$-cocycle condition of Definition~\ref{def:compatible2cocy}.
}
\end{remark}

\begin{remark}
	{\rm 
  In the case of Remark~\ref{rmk:mutuallydist}, when a system of compatible structures is defined on the same base space, a compatible system of $2$-cocycles is the same as a $2$-cocycle in the labeled cohomology of \cite{ESZ}. 
}
\end{remark}
Let $\{(Q_i,T_{ij})\}$ be a compatible system of ternary self-distributive structures and let $\{\alpha_{ij}\}$ be a compatible system of $2$-cocycles. Let $X_i = \mathbbm k\langle Q_i\rangle$ and denote the linearized maps $T_{ij}$ by the same symbols. To these data it is possible to associate twisting morphisms and braidings as follows. For each $X_i$ there is a twisting $\theta^{\alpha_{ij}}: X_i\otimes X_i\longrightarrow X_i\otimes X_i$ defined by extending the assignment
$$
x\otimes y\mapsto \alpha_{ii}(x,x,y)\alpha_{ii}(y,x,y)\cdot T_i(x,x,y)\otimes T_i(y,x,y).
$$
For each pair of objects $X_i$ and $X_j$ there is a braiding $c^{\alpha_{ij}}: X_i^{\otimes 2}\otimes X_j^{\otimes 2} \longrightarrow X_j^{\otimes 2}\otimes X_i^{\otimes 2}$ induced by
$$
(x\otimes y) \otimes (z\otimes w) \mapsto \alpha_{ij}(x,z,w)\alpha_{ij}(y,z,w)\cdot (z\otimes w)\otimes (T_{ij}(x,z,w)\otimes T_{ij}(y,z,w)).
$$

The construction of Definition~\ref{def:ribboncat} is adapted to this case and the corresponding category is denoted by $\mathcal{R}^*_{\{\alpha_{ij}\}}(\{X_i\})$.

\begin{theorem}\label{thm:compatibleribboncat}
	The category $\mathcal{R}^*_{\{\alpha_{ij}\}}(\{X_i\})$ with braiding and twisting morpshisms defined above is a ribbon category. Moreover, if two systems $\{\alpha_{ij}\}$ and $\{\beta_{i,j}\}$ are equivalent, then $\mathcal{R}^*_{\{\alpha_{ij}\}}(\{X_i\})$ and $\mathcal{R}^*_{\{\beta_{ij}\}}(\{X_i\})$ are equivalent.
\end{theorem}
\begin{proof}
	The proof is substantially the same as that of Theorem~\ref{thm:ribboncat}, where the $2$-cocycle condition of $\alpha$ is replaced by compatibility condition of the family $\{\alpha_{ij}\}$ and self-distributivity of $T$ is replaced by compatibility of the system $\{T_{ij}\}$. When $\{\alpha_{ij}\}$ and $\{\beta_{ij}\}$ are equivalent, one can construct maps $\tilde f_i$ as in Theorem~\ref{thm:ribboncat} and show that these induce an equivalence of ribbon categories. 
\end{proof}

Observe that $\mathcal{R}^*_{\{\alpha_{ij}\}}(\{X_i\})$ is a ribbon category with ``distinguished'' objects $X_i\otimes X_i$, and other objects given by tensor products obtained from the distinguished ones. 

\section{The Ribbon cocycle invariant is a quantum invariant}\label{sec:ribbonquantum}

The ribbon category $\mathcal{R}^*_{\alpha}(X)$ allows to define an invariant of framed links from any ternary self-distributive operation and a fixed ternary $2$-cocycle, following standard procedures \cite{Tur}. Let $(X,T)$ be a ternary self-distributive object arising from a set-theoretic ternary quandle $Q$ as above, and $[\alpha] \in H^2(Q,A) $ be fixed. A framed link is represented by the closure of an element $b \in FB_n$ of the framed braid group on $n$ ribbons \cite{KS} where, since twisting of the ribbon and crossings commute, it is assumed that the twists are on top of the braid. Using the same notation of \cite{KS}, $$b = t_1^{r_1}\cdots t_n^{r_n} \cdot \tau,$$
where $r_i$ are integers indicating the number of twists of the $i^{\rm th}$ ribbon and $\tau$ is an element of the braid group $B_n$. 

Then, the quantum invariant associated to $(X,T)$ and $\alpha$, is obtained by considering the object $X^2\boxtimes \cdots \boxtimes X^2$ ($n$-fold product) in $\mathcal{R}^*_{\alpha}(X)$ and taking the trace of the morphism $\Phi_b: X^2\boxtimes \cdots \boxtimes X^2 \longrightarrow X^2\boxtimes \cdots \boxtimes X^2$ corresponding to $b$ as follows. Each generator $t_i^{r_i}$ corresponds to $(\theta^\alpha)^{r_i}$, and each crossing $\sigma_i^{\pm 1}$ in the product defining $\tau$ corresponds to $(c^\alpha_{2,2})^{\pm 1}$. This invariant is denoted by the symbol
$$
\Psi_{\mathcal D} (X,T,\alpha),
$$
where $\mathcal D$ is the ribbon diagram representing $b$.

The following theorem establishes that two procedures described in Section~\ref{sec:ribboncocy} and this section coincide. 

\begin{theorem}\label{thm:quantum}
	Let $L$ be a framed link, with presentation given by a framed braid $b$ as above, $(X,T)$ be a ternary self-distributive structure that is linearized over $\mathbbm k$, and $\alpha$ a $2$-cocycle of $(X,T)$ with coefficitents in $A$. Suppose that $\chi: A\longrightarrow \mathbbm k^\times$ is a group character. Fix a diagram $\mathcal D$ of $L$. Then the ribbon cocycle invariant $\chi \Theta_{\mathcal D}(X,T,\alpha)$ and the quantum invariant $\Psi_{\mathcal D} (X,T,\chi\circ \alpha)$ coincide.	
	\end{theorem}
\begin{proof}
	The proof is very similar to that of Theorem~3.5 in \cite{Gra}. Observe that in order to compute $\Psi_{\mathcal D} (X,T,\alpha)$, one has to consider all combinations of basis vectors $x_{i_1}^{\epsilon(i_1)}\otimes \cdots\otimes x_{i_n}^{\epsilon(i_n)}$ apply the endomorphism $\Phi_b$ and therefore  apply it to all the possible combinations $x_{j_1}^{\epsilon(j_1)}\otimes \cdots \otimes x_{j_n}^{\epsilon(j_n)}$. It follows that the only nontrivial contributions to  $\Psi_{\mathcal D} (X,T,\alpha)$ are obtained when $x_{j_1}^{\epsilon(j_1)}\otimes \cdots \otimes x_{j_n}^{\epsilon(j_n)} = x_{i_1}^{-\epsilon(i_1)}\otimes \cdots \otimes x_{i_n}^{-\epsilon(i_n)}$, by definition of $ev$ and $coev$ maps in $\mathcal{R}^*_{\alpha}(X)$. At each crossing, whether corresponding to a twisting $t^{r_i}$ or one of the factors of the braid $\tau$, $\Phi_b$ contributes with a scalar given by either
	$$
	\alpha(z_1,z_1,z_2) \alpha(z_2,z_1,z_2),
	$$
	or
	$$
	\alpha(z_1,w_1,w_2) \alpha(w_2,w_1,w_2),
	$$
	where it has been assumed that $\theta^\alpha$ is applied to the vector $z_1\otimes z_2$ and $c^\alpha$ is applied to the vector $z_1\otimes z_2\boxtimes w_1\otimes w_2$. The vector output is given by either
	$$
	T(z_1,z_1,z_2)\otimes T(z_2,z_1,z_2),
	$$
	or
	$$
	w_1\otimes w_2\boxtimes T(z_1,w_1,w_2)\otimes T(w_2,w_1,w_2).
	$$
	It therefore follows that the only nontrivial contribution to $\Psi_{\mathcal D} (X,T,\alpha)$ corresponds to the colorings of $\mathcal D$ by $X$ and each contribution equals one of the summands that define $\Theta_{\mathcal D}(X,T,\alpha)$. The proof is complete. 
\end{proof}

The quantum invariant  $\Psi_{\mathcal D} (X,T,\alpha)$ does not only provide a differennt interpretation of the ribbon cocycle invariant $\Theta_{\mathcal D}(X,T,\alpha)$, but it is also suitable for a generalization to framed braids (not necessarily closed). 
\begin{theorem}
	Let $F$ be a diagram of a framed braid $b \in FB_n$. Then $\Phi_b$ defined as above is an invariant of $F$. 
\end{theorem}
\section{Examples and Computations}\label{sec:examples}
\subsection{Examples from compositions of binary quandles}

In this subsection it is showed that the notions of $G$-family of quandles and $G$-family of $2$-cocycles \cite{IIJO,Nos} provide interesting examples of compatible systems of TSD structures. An explicit example using Nosaka cocycles is also shown in detail. 

Recall first, that a $G$-family of quandles is a set $X$ along with a family of quandle operations $*^g$ indexed by a group $G$ (i.e. $g\in G$ for all $g$), and satisfying the compatibility conditions
\begin{eqnarray*}
(x*^gy)*^hy &=& x*^{gh}y,\\
(x*^gy)*^hz &=& (x*^hz)*^{h^{-1}gh}(y*^hz).
\end{eqnarray*}
Given a $G$-family of quandles one can construct a compatible system of ternary self-distributive operations as follows. Linearize the base set $X$ to obtain $V = \mathbbm k\langle X\rangle$. Then define maps $T_{gh} : V\otimes V\otimes V \longrightarrow V$ by the assignment $x\otimes y\otimes z \mapsto (x*^hy)*^{h^{-1}}z$. A direct computation shows that the system so defined is indeed compatible. Observe that  $T_{gg}$ being self-distributive is an instance of the fact that composing mutually distributive binary operations produces a ternary self-distributive operation, as in \cite{ESZ}.

\begin{definition}\label{def:Gfamilycompatible}
	{\rm 
The compatible system constructed above from the $G$-family of quandles $(X,*^g)_{g\in G}$ is called the {\it compatible system associated to a $G$-family}.	
}
\end{definition}

Since the main construction of Section~\ref{sec:ribboncat} associate a ribbon category to a compatible system of distributive structures by means of a system of $2$-cocycles, it is fundamental to obtain such objects for compatible system associated to $G$-families. As shown in the next result, it is possible to do so using the notion of $G$-family $2$-cocycles, see \cite{IIJO,Nos} for the definition of $G$-family (co)homology. In what follows it is assumed that the $X$-set $Y$ appearing in $G$-family cohomology is a singleton endowed with the trivial $G$-family action. The set $Y$ is therefore omitted without further notice, but this should not cause any difficulties. 

\begin{proposition}\label{prop:Gcocy}
	Let $(X,*^g)_{g\in G}$ and $(V,T_{gh})_{g,h\in G}$ be a $G$-family of quandles and the associated compatible system of self-distributive structures, respectively. Let $\theta$ be a $G$-family $2$-cocycle. Then there is an associated compatible system of $2$-cocycles $\theta_{gh}$ defined as follows
	$$
	\theta_{gh}(x\times y\times z) := \theta((x,e)\times (y,g)) + \theta((x*^gy,e)\times (z,h)).
	$$
\end{proposition}
\begin{remark}\label{rmk:Gcocy}
	{\rm 
A couple of observations are due. Firstly, notice that since $X$ acts trivially on $Y$, the two terms corresponding to deleting the first entry of $\theta$, according to the definition of $G$-family $2$-cocycle cancel each other, so that it is reasonable to arbitrarily choose on element of $G$ to label all the first entries $x$, where the obvious choice falls upon the neutral element $e$ of $G$. Secondly, there is a parallel between labels assignment in the definition of chain maps in ``labeled cohomology'' of Theorem~5.3 in \cite{ESZ}, particularly clear from Remark~5.7 in the same article, and the group element enriching $y$ or $z$ to a pair $(y,g)$ or $(z,h)$, respectively. 
}
\end{remark}

\begin{proof}[Proof of \ref{prop:Gcocy}]
	Using Remark~\ref{rmk:Gcocy}, the proof is almost immediate. In fact, the condition that $\theta_{gh}$ has to satisfy is (in additive notation)
	\begin{eqnarray*}
	\lefteqn{\theta_{fg}(x\times y\times z) + \theta_{fh}(T_{ij}(x\times y\times z)\times u\times v)}\\
	&=& \theta_{fh}(x\times u\times v) + \theta_{fg}(T_{fh}(x\times u\times v),T_{gh}(y\times u\times v)\times T_{gh}(z\times u\times v)).
\end{eqnarray*}
Using the definition of $\theta_{gh}$, the $G$-family $2$-cocycle condition becomes equivalent to labeled cohomology $2$-cocycle condition, since terms obtained by deleting $x$ cancel. Now the same proof as in Remark~5.7 of \cite{ESZ} can be applied, mutatis mutandis, to complete. 
\end{proof}

Let now $G = {\rm SL}(2;\Z_3)$ and $X = \Z_3\times \Z_3$, with operations $\{*^g\}_{g\in G}$ defined as follows \cite{IIJO}, $x*^gy := gx + (\mathbbm 1-g)y$ for all $x,y\in X$ and $g\in G$, where $G$ acts on $X$ by matrix multiplication on column vectors, and $\mathbbm 1$ is the identity of $G$. A direct computation shows that these operations define a $G$-family structure on $(X,G)$ (Proposition~2.3 in \cite{IIJO}). From this data, Nosaka has constructed \cite{Nos} a $G$-family $2$-cocycle, that has been employed in \cite{IIJO} to compute cocycle invariants of certain handlebody knots and distinguish them from their mirror images. As pointed out above, it is not restrictive to omit the singleton set $Y := \{y\}$ in the original construction. Define $\alpha : (X\times G)^{\times 2} \longrightarrow \Z_3$ by
$$
(x,g)\times (y,h) \mapsto \lambda(g) {\rm det}(x-y, (1-h)^{-1}y), 
$$
where $\lambda$ is the abelianization function, defined by $\lambda (A) = (a+d)(b-c)(1-bc)$, for a matrix 
$A := \left( \begin{array}{cc}  a&b\\ c&d \end{array}\right)$. Then $\alpha$ is a $G$-family $2$-cocycle and it follows that $\{\alpha_{gh}\}$ defined as in Proposition~\ref{prop:Gcocy} is a compatible system of $2$-cocycles.

\begin{example}
	{\rm 
Applying Theorem~\ref{thm:compatibleribboncat} to the compatible system of $2$-cocycles $\{\theta_{gh}\}_{g,h\in G}$ associated to Nosaka's $G$-family $2$-cocycle $\alpha$ via Proposition~\ref{prop:Gcocy} one obtains a ribbon category. The braiding at level $2$ is given explicitly by the maps
\begin{eqnarray*}
c^{\alpha_{gh}}_{2,2}(x\otimes y\boxtimes z\otimes w) &=& \alpha((x,e),(z,g)\alpha((x*^gz,e),(w,h))\\
&& \hspace{1cm} \alpha((y,e),(z,g)\alpha((x*^gz,e),(w,h))\\
&&\cdot z\otimes w\boxtimes (x+(h^{-1}-\mathbbm{1})z+(\mathbbm{1}-h)w)\otimes \\
&&\hspace{1cm} (y+(h^{-1}-\mathbbm{1})z+(\mathbbm{1}-h)w),
\end{eqnarray*}
while twisting morphisms are given by 
\begin{eqnarray*}
	\theta^{\alpha_{gh}} (x\otimes y)&=& \alpha((x,e),(x,g)\alpha((x,e),(y,h))\\
	&& \hspace{1cm} \alpha((y,e),(x,g)\alpha((y*^gx,e),(w,h))\\
	&& \cdot (h^{-1}x+ (\mathbbm{1}-h^{-1})y)\otimes ((h^{-1}-\mathbbm{1})x- h^{-1}y).
	\end{eqnarray*}
}
\end{example}

\subsection{Examples from heap structures}\label{subsec:heap}
Recall that given a group $G$, the heap operation $G\times G\times G\longrightarrow G$ defined by $x\times y\times z \mapsto xy^{-1}z$, defines a ternary quandle structure on $G$. Linearizing this assignment over a field $\mathbbm k$ produces a ternary quandle object in the category of vector spaces, where the diagonal map is induced by $x\mapsto x\otimes x\otimes x$. This definition in fact coincides with the quantum heap of the Hopf algebra structure on the group ring $\mathbbm k [G]$, as it can be seen directly. 

\begin{example}
	{\rm 
Let $\Z_2$ be the cyclic group of order $2$ and let $\mathbb C[\mathbb Z_2]$ be the structure defined above, obtained by linearizing the heap operation of $\Z_2$. The elements of $\mathbb Z_2$ are identified with the symbols $e_x$, with $x\in \mathbb Z_2$, generating the two dimensional vector space. A direct computation shows that $H^2(\Z_2,\Z_2) = \Z_2\oplus \Z_2$ with generators corresponding to the equivalence classes of characteristic functions $\chi_{(0,0,0)}$ and $\chi_{(0,1,1)}$. Fix the cocycle $(0,1)\in H^2(\Z_2,\Z_2) $, i.e. the map $\phi(x,y,z) = 1$ if $(x,y,z) = (0,1,1)$ and $\phi(x,y,z) = 0$ otherwise. Identifying $\Z_2$ with $(-1)^i\in \C^{\times}$ one gets a nontrivial cohomology class in $H^2(\Z_2,\C^{\times})$, still denoted by $\phi$. The corresponding twisting morphisms and braiding morphism are as follows
\begin{eqnarray*}
	\theta^{\phi} (e_x\otimes e_y) &=& (-1)^{\phi(x,x,y) + \phi(y,x,y)} \cdot e_{x+x+y}\otimes e_{y+x+y}\\
&=& e_y\otimes e_x,
 	\end{eqnarray*}
 that is, the twisting morphism is given by transposition, and 
 \begin{eqnarray*}
 	c^\phi(e_x\otimes e_y\boxtimes e_z\otimes e_w) &=& (-1)^{\phi(x,z,w)+\phi(y,z,w)} \cdot e_z\otimes e_w\boxtimes e_{x+z+w}\otimes e_{y+z+w}\\
 	&=& \begin{cases}
 		 e_z\otimes e_w\boxtimes e_x\otimes e_y\ \  z = w \\
 		e_z\otimes e_w\boxtimes e_{x+1}\otimes e_{y+1} \ \ {\rm otherwise}
 	\end{cases} 
 	\end{eqnarray*}
}
\end{example}

\begin{example}\label{ex:abelianheap}
	{\rm 
Let $X$ be the abelian group $\Z_n$ with group heap structure, and set $A = \Z$ taken with multiplicative notation, with generator $g$. Suppose $\rho$ is a given group character mapping $A$ in the group of units of $\mathbbm k$. In \cite{SZ}, Lemma~3.7, it is shown that the $2$-cochain $\phi_i: X^{\times 3} \longrightarrow \Z$ defined by the formula
$$
\phi_i := \sum_{x\in \Z_n} [\sum_{j=0}^{n-1} \chi_{(x,j,j+i)}],
$$
where $\chi_{(x,y,z)}$ is the characteristic function at the triple $(x,y,z)\in X^{\times 3}$ is a nontrivial $2$-cocycle for any choice of $i = 1, \ldots , n-1$. It is in fact thereby proved that $[\phi_i] \neq [\phi_k]$ in the second cohomology group, whenever $i\neq k$ in $\Z_n$. The ribbon category corresponding to $\phi$, for some choice of $n\in \N$ and $0 \neq i\in\Z_n$ is determined by braiding and twisting morphisms obtained as follows. For all $e_x,e_y,e_z,e_w\in \mathbbm k[X]$, the linearization of $X$ over $\mathbbm k$ coinciding with the group algebra of $\Z_n$, $c_{2,2}^{\phi_i}$ maps simple tensors according to the assignment 
\begin{eqnarray*}
c_{2,2}^{\phi_i}(x\otimes y\boxtimes z\otimes w) &=& \rho(g^{\phi_i(x,z,w)})\rho(g^{\phi_i(y,z,w)})e_z\otimes e_w\boxtimes e_{x-z+w}\otimes e_{y-z+w}\\
&=& \rho(g^{2\phi(x,z,w)})e_z\otimes e_w\boxtimes e_{x-z+w}\otimes e_{y-z+w}\\
&=& \begin{cases}
	\rho(g^2) e_z\otimes e_w\boxtimes e_{x+i}\otimes e_{y+i} \ \ {\rm if}\ w-z = i\\
	e_z\otimes e_w\boxtimes e_{x+k}\otimes e_{y+k}\ \ {\rm if}\ w-z=k\neq i
\end{cases}
\end{eqnarray*}
The twisting morphism $\theta_2^{\phi_i}$ maps simple tensors as
\begin{eqnarray*}
\theta_2^{\phi_i}(e_x\otimes e_y) &=& \rho(g^{\phi_i(x,x,y)})\rho(g^{\phi_i(y,x,y)})e_{y}\otimes e_{2y-x}\\
&=& \rho(g^{2\phi_i(x,x,y)})e_{y}\otimes e_{2y-x}\\
&=&\begin{cases}
	\rho(g^2) e_y\otimes e_{y+i} \ \ {\rm if}\ y-x = i\\
	e_y\otimes e_{y+k} \ \ {\rm if}\ y-x = k\neq i 
\end{cases}
\end{eqnarray*}
We note that the twisting morphism $\theta_2^{\phi_i}$ is determined, up to scalar multiplication, by the Takasaki quandle operation $x*y = 2y-x$ associated with the abelian group $\Z_n$. This is in fact a general feature of the twisting morphism of an abelian heap. It is also easy to see that for non-abelian heaps one obtains the core of a quandle, instead of the Takasaki structure. Generalization of the preceding braiding and twisting structure to the case of linear combinations of $\phi_i$'s is easily obtained from the previous equations. 
}
\end{example}

\begin{example}
	{\rm
Let $X = D_3$ be the dihedral group on $6$ elements, with presentation $\langle s,r\ |\ s^2 = r^3 =1,\ srs = r^{-1}\rangle$. Once again consider the group heap strucutre on $X$ and linearize it to obtain a quantum heap on the group ring $\mathbb C[X]$. Denote characteristic functions $\chi_{(x,y,z)}$, those cochains with coefficients in $\Z_3$ defined by $\chi_{(x,y,z)}(x',y',z') = 1$ if $(x',y',z') = (x,y,z)$, and $0$ otherwise. Define the $2$-cochain 
$$
\psi = \sum_x [\chi_{(x,1,r)}+\chi_{(x,r,r^2)}+\chi_{(x,r^2,1)}] + \sum_x [\chi_{(x,s,sr)}+\chi_{(x,sr,sr^2)}+\chi_{(x,sr^2,s)}].
$$
A direct computation shows that $\psi$ satisfies the $2$-cocycle condition and it is therefore a $2$-cocycle. moreover $\psi$ is nontrivial \cite{SZ}, Example~5.13 and Proposition~5.14, so that $[\psi] \neq 0$. Mapping $\Z_3$ to the $3^{\rm rd}$ roots of unity $G_3$ we obtain $\psi\in Z^2(\Z_3,G_3)$, where $G_3$ acts on $\C [\Z_3]$, and therefore on $\C [\Z_3]\otimes \C [\Z_3]$, by scalar multiplication. Twisting morphisms are obtained as
\begin{eqnarray*}
\theta^\psi (x\otimes y) &=& e^{\frac{2\pi i}{3}(\psi(x,x,y)+\psi(y,x,y))}\cdot y\otimes yx^{-1}y,
\end{eqnarray*}
where the scalar $e^{\frac{2\pi i}{3}(\psi(x,x,y)+\psi(y,x,y))}$ is nontrivial if and only if $x^{-1}y = r$, in which case we obtain $e^{\frac{2\pi i}{3}(\psi(x,x,y)+\psi(y,x,y))} = e^{\frac{4\pi i}{3}}$. Observe that $ yx^{-1}y$ is the core quandle operation. In fact, $\theta^\psi$ corresponds (up to the multiplying scalar) to the $R$ matrix obtained by linearizing the set-theoretic solution of Yang-Baxter equation corresponding to the core of group $D_3$. The braiding $c_{2,2}^\psi$ is given explicitly by 
\begin{eqnarray*}
c^\psi_{2,2}(x\otimes y\boxtimes z\otimes w) &=& e^{\frac{2\pi i}{3}(\psi(x,z,w)+\psi(y,z,w))}  \\
&&\hspace{1cm} \cdot z\otimes w \boxtimes xz^{-1}w\otimes yz^{-1}w,
\end{eqnarray*}
where the scalar multiple is nontrivial if and only if $z^{-1}w = r$, in which case we have $e^{\frac{2\pi i}{3}(\psi(x,z,w)+\psi(y,z,w))} = e^{\frac{4\pi i}{3}}$, similarly to the case of twisting. 
}
\end{example}

\subsection{Computations of quantum invariants}

In this subsection we illustrate the theory by computing some invariants of a few basic framed links corresponding to linearized ternary self-distributive structures and $2$-cocycles given above. We compare with the invariants computed in \cite{SZ}. In fact it will be apparent how the two types of invariants encode the same information, as proved in Theorem~\ref{thm:quantum}.

\begin{example}
	{\rm 
Let us consider the unknot with frame $n\in \N$. Let $H = \Z_m$ for some arbitrary $m$, considered with abelian heap structure, i.e. $T(x,y,z) = x- y +z$. As already argued in Example~\ref{ex:abelianheap}, for each $i = 1, \ldots , m-1$ there exists a nontrivial cocycle with coefficients in $\Z$ given by $\phi_i := \sum_{x\in \Z_m} [\sum_{j=0}^{m-1} \chi_{(x,j,j+i)}]$. Let us fix an arbitrary $i$ and let $\zeta\in \C$ denote a primitive $m^{\rm th}$-root of unity. We assume $\Z$ to be generated (in multiplicative notation) by $g$, and define the character $\tilde \chi: \Z \longrightarrow \C^\times$ by reducing the exponents of $g$ modulo $m$ and identifying it with the corresponding power of $\zeta$. The map that is associated to the $n$-framed unknot is $\theta^n$, where we set $\theta := \theta^{\phi_i}$ for simplicity. We denote the basis vectors of the vector space $X = \mathbbm k\langle H\rangle$ generated by $H$ by the symbols $e_x$, for $x\in H$. Then $\theta^n$ is easily seen to be given by the map $e_x\otimes e_y \mapsto q(x,y)\cdot e_{ny-(n-1)x}\otimes e_{(n+1)y - nx}$, where the unit $q(x,y)\in \C^\times$ is determined below. For a vector to contribute to the trace of $\theta^n$ we need $e_x = e_{ny-(n-1)x}$ and $e_y = e_{(n+1)y - nx}$, which gives $n(y-x) = 0$. If $(n,m) = 1$ then there are $| X |$ vectors that satisfy this condition. We have that $tr_q (\theta^n) = \mid X \mid$, since $i\neq 0$. Observe that the condition of $e_x\otimes e_y$ contributing to the trace of $\theta^n$ can be rephrased as $x,y$ giving a coloring of the diagram of the $n$-framed unknot. So, when choosing $m$ coprime with $n$ the invariant is simply counting colorings. If $(n,m) = d\neq 1$, then we have $d$ elements $\alpha$ divisible by $m/d$. Each of these solutions gives $|H| = m$ vectors that contribute nontrivially to the trace of $\theta^n$. Moreover, whenever $i = \alpha$ for one of the previous solutions, we have a contribution of $\zeta^{2n}$. We have
$$
tr_q(\theta^n) = \begin{cases}
dm  \ \ \ (i,m/d) \neq m/d\\
(d-1)m + \zeta^{2n} m \ \ \ (i,m/d) = m/d
\end{cases}
$$
The cocycle invariant $\Psi$, by direct computation, is seen to be equal to $m e\otimes e$ when $m$ and $n$ are coprime, and 
$$
\Psi = \begin{cases}
md\cdot  e\otimes e \ \ \ (i,m/d) \neq m/d\\
m( d-1)\cdot e\otimes e + mg^n\otimes g^n \ \ \ (i,m/d) = m/d
\end{cases}
$$
when $(m,n) = d$. We see that applying $\chi$ to $\Psi$ we obtain $tr_q(\theta^n)$, as required. Moreover, we can choose $m$ and $i$ such that $tr_q(\theta^n)  = dm + \zeta^{2n} m$ and $\zeta^{2n} \neq 1$, so the invariant detects twisting. 
}
\end{example}

\begin{example}
	{\rm 
	Let us now consider the torus link $T(2,2n)$ on two strings, with even number of crossings. We compute the quantum invariant corresponding to the cocycle $\phi_i\in Z^2(\Z_m,\Z)$, for some $i = 1, \ldots, n-1$. Set $X = \C\langle \Z_m\rangle$. The framed braid whose closure gives $T(2,2n)$ corresponds to the endomorphism $c^{2n} : X^{\otimes 2} \boxtimes X^{\otimes 2} \longrightarrow X^{\otimes 2} \boxtimes X^{\otimes 2}$, obtained by composing the braiding of two ribbons $2n$ times. We use the symbol $\boxtimes$ to distinguish pairs corresponding to the two edges of a ribbon, following previous conventions. We choose, as before, an integer $m$ and a primitive $m^{\rm th}$-root of unitiy $\zeta$ and we use again the map $\tilde \chi$ that sends a generator, say $g$, of $\Z$ in multiplicative notation to $\zeta$. On basis vectors we have, by direct computation, $c^{2n}(e_x\otimes e_y \boxtimes e_z\otimes e_w) = q(x,y,z,w)\cdot e_{x+n(w-z)}\otimes e_{y+n(w-z)}\boxtimes e_{z+n(y-x)}\otimes e_{w+n(y-x)} = q(x,y,z,w)\cdot  e_x\otimes e_y \boxtimes e_z\otimes e_w$, where $q(x,y,z,w)\in \C^\times$ is determined as follows. We shorten notation by omitting the variables, therefore writing just $q$. When $y-x = w-z = i$ we have $q = \zeta^{4n}$, when just one of $w-z$ or $y-x$ equals $i$, $q = \zeta^{2n}$, while $q = 1$ whenever both $w-z$ and $y-x$ are different from $i$. Then $c^{2n}$ is diagonal and each vector contributes to the trace so, to complete the computation, one needs just to count how many vectors contribute with either of the weights $q$ given above. We have $tr_q(c^{2n}) = n^2\zeta^{4n} + 2n(n-1) \zeta^{2n} + n^4 + n$. Observe that the computation parallels perfectly that of Example~5.12 in \cite{SZ}, where the fact that each initial arc coloring defines a full coloring of $T(2,2n)$ translates into the statement that $c^{2n}$ is diagonal, and where in the cocycle invariant the weights contribute to the entries of the tensor product of the group algebra $\Z_n[\Z]$ depending on the components of the link. 
}
\end{example}
\section{Generalized construction}\label{sec:generalized}

In this section, we develop a generalized version of the theory described above, where we consider symmetric monoidal categories. To this objective, we first need to provide a framework for ternary self-distributivity in symmetric monoidal categories and introduce a categorical version of $2$-cocycle condition. The first construction has been introduced by the author, along with M. Elhamdadi and M. Saito, in \cite{ESZ}, while the second construction will be introduced herein. 

Our main interest in generalizing the construction from the linearized case described above to more general objects in symmetric monoidal categories lies in the following discussion. When computing 
the quantum invariant associated to a TSD set we see that a vector contributes to the invariant, i.e. to the (quantum) trace of the linear map $X^{\otimes 2n} \longrightarrow X^{\otimes 2n}$ associated to a framed braid diagram with n doubled-strings, if the coloring condition on the edges of the diagram is satisfied. This is due to the fact that the comultiplication of $X$ is simply given by producing two copies of an element of $X$. For a general TSD object, say in the category of vector spaces, the comultiplication is usually not the same as $x\mapsto x\otimes x$, so that there can be contributions to the quantum trace that do not correspond to colorings of the framed diagram by $X$. We expect that this phenomenon produces stronger invariants than the ribbon cocycle invariant. In fact, while a trivial cocycle does not produce new non-trivial invariants in the linearized case (it simply counts the colorings of a diagram, which is a known invariant), with a general comultiplication we can obtain nontrivial invariants corresponding to the YB operator associated to the TSD object even when this is not deformed by a nontrivial cocycle. 

Let $\mathcal C$ be a symmetric monoidal category, with tensor functor denoted by $\otimes$ and let $X$ be a fixed object in $\mathcal C$. Then, associating the switching morphism $\tau: X^{\otimes 2} \longrightarrow X^{\otimes 2}$ to the transposition $(1 2)\in \mathbb S_2$ we obtain a ``representation'' of the infinite symmetric group $\mathbb S_{\infty}$ as follows. Let $\sigma\in \mathbb S_n$ for some $n\in \mathbb N$, we decompose the permutation in a product of transpositions $\sigma = \sigma_1\cdots \sigma_k$ for some $k$. Then we define the corresponding automorphism of $X^{\otimes n}$ to be $\tau_1\circ \cdots \circ \tau_k$, where $\tau_i : X^{\otimes n} \longrightarrow X^{\otimes n}$ is the automorphism given by $\mathbbm 1^{i-1} \otimes \tau_{X,X}\otimes \mathbbm 1^{n-i+1}$. Then it is verified that this assignment does not depend on the choice of decomposition of a permutation into transpositions, and it is therefore well defined. We obtain a correspondence $\shuffle : \mathbb S_{\infty} \longrightarrow \cup_n {\rm Hom}(X^{\otimes n},X^{\otimes n})$, and $\shuffle_\sigma$ is the automorphism of $X^{\otimes n}$ that $\shuffle$ associates to the permutation on $n$ elements $\sigma$. In the rest of the section we will say that $\shuffle_\sigma$ is the morphism corresponding to the permutation $\sigma$. 
\subsection{TSD objects in symmetric monoidal cateogries}
Let $(\mathcal C, \otimes, \tau)$ be a symmetric monoidal category and let $X$ be a ternary self-distributive (TSD) object in $\mathcal C$, i.e. a comonoid object with a morphism $T: X\otimes X\otimes X \longrightarrow X$ satisfying categorical self-distributivity and commuting with the diagonal morphism $\Delta_3 := \Delta\circ (\Delta\otimes \mathbbm 1)$. Specifically, categorical self-distributivity means that $X$ is endowed with a morphism $\Delta : X \longrightarrow X\otimes X$ which makes the same diagrams of coassociativity commute, a morphism $\epsilon : X\longrightarrow \mathbbm k$ which satisfies the same diagrams of counit in a coalgebra, and the morphism $T$ makes the following diagram commute (case $n=3$ in \cite{ESZ})
\begin{center}
	\begin{tikzcd}
		&X^{\otimes 9}\arrow{dl}[swap]{\shuffle_t} & &X^{\otimes 5}\arrow{ll}[swap]{\mathbbm 1^{\otimes 3}\otimes \Delta_3^{\otimes 2}}\arrow {rd}{T\otimes \mathbbm{1}^{\otimes 2}} &  \\
		X^{\otimes 9}\arrow{dd}[swap]{T\otimes T \otimes T} & & & & X^{\otimes n}\arrow{dd}{T}\\
		& & & &\\
		X^{\otimes 3}\arrow{rrrr}[swap]{T}& & & &X 
	\end{tikzcd} 
\end{center} 
where $\Delta_3 := (\Delta\otimes \mathbbm{1})\circ \Delta = (\mathbbm{1}\otimes \Delta)\circ \Delta$, and $\shuffle_t$ is the shuffle corresponding to ternary self-distributivity 
$$
t =  (2,4)(3,5)(6,8)(3,7).
$$
 In the rest of the article we will denote $\Delta_n := (\Delta\otimes \mathbbm 1^{n-1})\circ\cdots \circ(\Delta\otimes \mathbbm 1)\circ\Delta : X \longrightarrow X^{\otimes n}$, to indicate the $n$-fold diagonal of the comonoid object $X$. 


\begin{remark}
	{\rm 
In what follows we will assume our TSD objects to be cocommutative (as comonoids), as the main proof of the section will make use of this assumption. We point out that the preliminary definitions and results make sense without this further assumption, unless otherwise specified. 
}
\end{remark}

The final objective of this section is to generalize the construction of ribbon category from a TSD operation to the case of TSD objects in symmetric monoidal categories. When we linearize a TSD operation $T$ over a field (or ring), as we have seen in Section~\ref{sec:ribboncat}, the rack axiom of $T$ which states that $T(x,y,z) = d$ has a solution for $x$ for every $d$, once we fix two elements $y,z$, automatically provides a way of defining an inverse to the braiding morphisms. In the setting of symmetric monoidal categories, the definition of TSD object given above does not provide any condition guaranteeing the possibility of introducing an inverse to the braiding morphism. The next definition provides an answer to this issue. 
\begin{definition}
	{\rm 
Let $X$ be a TSD object in a symmetric monoidal category $\mathcal{C}$. Then a {\it ternary rack} in $\mathcal{C}$ is a comonoid object $X$ in $\mathcal{C}$ together with a pair of morphisms $T, T^{-1}: X^{\otimes 3} \longrightarrow X$ satisfying the TSD condition given above, along with the equation
$$
T^{-1}\circ [T\otimes \mathbbm 1^{\otimes 2}] \circ \mathbbm 1\otimes \Delta\otimes \Delta = \mathbbm 1 \otimes \epsilon\otimes \epsilon,
$$	
where equality is meant as an equality of morphisms $X^{\otimes 3} \longrightarrow X$ in $\mathcal C$. We also require an analogous equation to be satisfied, where the roles of $T$ and $T^{-1}$ are exchanged.
}
\end{definition}

\begin{example}\label{ex:qheap}
	{\rm 
The fundamental example of TSD object in the symmetric monoidal category of vector spaces is that of an involutory Hopf algebra with {\it quantum heap} operation. This is a ``categorical'' version of the notion of heap of a group whose cocycle invariants have been introduced and studied in \cite{SZ}. The quantum heap operation is given by extending the operation $x\otimes y \otimes z \mapsto xS(y)z$ by linearity. A similar construction holds replacing a Hopf algebra by a Hopf monoid in a symmetric category $\mathcal C$. This means that $H$ is a bimonoid, i.e. an object that is both monoid and comonoid, and it is endowed with a morphism $s: H\longrightarrow H$ that satisfies the same commuative diagrams for the antipode as in the usual definition of Hopf algebra. It has been proved in \cite{heap} that an involutory Hopf monoid gives rise to a TSD object by a generalization of the quantum heap construction (see Theorem~7.12 therein). Both vector space and symmetric monoidal category versions of the proofs utilize the fact that 
$$
T\circ [T\otimes \mathbbm 1^{\otimes 2}]\circ[ \mathbbm 1\otimes \hat\tau]\circ\mathbbm{1} \otimes \Delta \otimes \Delta = \mathbbm 1\otimes \epsilon\otimes \epsilon,
$$
where $\hat \tau$ is the morphism $X^{\otimes 4}\longrightarrow X^{\otimes 4}$ corresponding to the permutation  $\bigl(\begin{smallmatrix}
1 & 2& 3& 4\\
1&3&4&2
\end{smallmatrix}\bigr)$. See for instance the first part of the proof of Proposition~7.10, and Lemma~7.11 in \cite{heap}. If $X$ is the heap object corresponding to an involutory Hopf monoid in $\mathcal C$, i.e. a TSD object, we have a ternary rack object in $\mathcal C$ by taking $T^{-1}$ to be $T\circ [\mathbbm 1\otimes \tau]$, where $\tau$ denotes the switching morphism of $\mathcal C$. Hopf algebras (or monoids) naturally give rise to TSD objects as well as ternary racks. More generally, one can replace the notion of involutory Hopf monoid by the more general one of (categorical) heap, of which Hopf monoids provide a fundamental example. See Definition~7.1 in \cite{heap}. 
}
\end{example}

\subsection{Examples of TSD objects}

There is a very natural situation in which TSD objects arise in a symmetric monoidal category, as described in the following example, which generalizes Example~\ref{ex:qheap}

\begin{example}\label{ex:hopfmonoid}
	{\rm 
		We describe explicitly the quantum heap TSD morphism in symmetric monoidal categories mentioned above.
		Let $X$ be an involutory Hopf monoid in $\mathcal C$. Then it has been shown in \cite{heap} that $X$ can be endowed with a morphism $T: X^{\otimes 3} \longrightarrow X$ that turns it into a TSD object in $\mathcal C$ as follows. 
		Set $T := \mu \circ (\mu \otimes \mathbbm 1)\circ (\mathbbm 1 \otimes s\otimes \mathbbm 1)$, where $\mu$ is the multiplication morphism and $s$ is the (involutive) antipode in $X$. Observe that when $X$ is a Hopf algebra in the category of vector spaces, we have that $(X,T)$ coincides with the quantum heap of Example~\ref{ex:qheap}. In fact, if in addition $X$ is the group algebra of some group $G$, then $T$ is the linearization of set-theoretic heap structure of $G$, as in Subsection~\ref{subsec:heap}, so that this is a natural generalization of objects already encountered. These structures will also be referred to as {\it quantum heaps}, as in the vector space case. 
	}
\end{example} 

\begin{example}\label{ex:qconjugation}
	{\rm 
	Let $X$ denote a Hopf monoid in $\mathcal C$. Then, we can define an operation akin to conjugation in a group, by means of the antipode $S$ of $X$. We define $q = \mu\circ (\mathbbm 1\otimes \mu) \circ  (S\otimes \mathbbm 1^{\otimes 2}) \shuffle \circ(\mathbbm 1\otimes \Delta)$, where $\shuffle$ corresponds to the transposition $(12)\in \mathbb S_3$. In the category of vector spaces or modules, i.e. when dealing with a Hopf algebra in the usual sense, this operation takes the form $x\otimes y\mapsto S(y^{(1)})xy^{(2)}$, where juxtaposition denotes multiplication in $X$. We refer to this operation as {\it quantum conjugation}. Iterating the operation $q$, i.e. defining $T = q\circ (q\otimes \mathbbm 1)$ we obtain a ternary operation, called {\it double quantum conjugation}. In \cite{ESZ} Section~8, it is seen that $q$ is binary self-distributive, and it is also proven that composing binary self-distributive operations yields a TSD operation. 
}
\end{example}

\begin{example}\label{ex:klinear}
	{\rm 
		Let $\mathcal C$ denote the category of vector spaces over a field $\mathcal k$, and let $L$ denote a Lie algebra. Set $X = \mathbbm k\oplus L$ and, denoting its elements by pairs $(a,x)$, we can define a coproduct on $X$ by the assingment 
		$$
		(a,x) \mapsto (a,x)\otimes (1,0) + (1,0)\otimes (0,x),
		$$
		and counit $\epsilon (a,x) = a$. It is easy to see that this structure defines a coalgebra in $\mathcal C$.  Then, we apply Example~8.8 in \cite{ESZ} defining $T$ by iteration of Lie bracket structure on basis vectors $(a,x)$ by the assignment
		$$
		(a,x)\otimes (b,y)\otimes (c,z) \mapsto (abc, bcx + c[x,y] + b [x,z] + [[x,y],z]).
		$$ 
		Then $T$ turns $X$ into a TSD object in the category of vector spaces (proof in \cite{ESZ}, Theorem~8.6 or Appendix A).  
	}
\end{example}

In fact, it can be proved that the TSD structures of Example~\ref{ex:klinear} are invertible, and therefore are ternary racks. We will not consider these objects in detail, leaving the proof of the previous claim to subsequent work, as they are not cocommutative. 
 
\subsection{Categorical 2-cocycle condition}

	Let us now consider an $\mathbb I$-linear symmetric monoidal category $\mathcal C$, where $\mathbb  I$ denotes the unit object of $\mathcal C$. Suppose $X$ is a unitary comonoid object in $\mathcal C$. This means that $X$ is endowed with morphisms $\Delta : X\longrightarrow X\otimes X$, $\epsilon : X\longrightarrow \mathbb I$ and $\eta: \mathbb I\longrightarrow X$ that make  commute the diagrams defining a unitary coalgebra in the category of vector spaces. In this situation, we say that a morphism $\alpha: X\longrightarrow \mathbb I$ is {\it convolution invertible} if there exists a morphism $\alpha^{-1}$ making the following diagram commute
	\begin{center}
	\begin{tikzcd}
		X^{\otimes 3}\arrow[rr,"\shuffle \Delta^{\otimes 3}"]\arrow[d,swap,"\epsilon^{\otimes 3}"] && X^{\otimes 3}\otimes X^{\otimes 3}\arrow[d,"\alpha\otimes \alpha^{-1}"]\\
		\mathbb I\arrow[rr,equal]&& \mathbb I\otimes \mathbb I
		\end{tikzcd}
	\end{center}
	where $\shuffle$ reorders the outputs of $\Delta^{\otimes 3}$ as in the comultiplication of a tensor coalgebra structure.  
	\begin{definition}\label{def:catcocy}
	{\rm
Let $X$ be a TSD object in a symmetric monoidal category. Then a convolution invertible morphism $\alpha$ is called {\it categorical 2-cocycle} with coefficients in $\mathbb I$ if the diagram 
 \begin{center}
 	\begin{tikzcd}
 		X^{\otimes 5}\arrow[rr,"\shuffle_1\circ(\Delta^3\mathbbm 1^2)"]\arrow[d,swap,"\shuffle_2\circ(\Delta\mathbbm 1^2\Delta_2^2)"] && X^{\otimes 8}\arrow[rr,"\alpha\alpha\circ(\mathbbm 1^3T\mathbbm 1^2)"]&&\mathbb I^{\otimes 2}\arrow[d,equal]\\
 	X^{\otimes 12}\arrow[rr,swap,"\alpha\alpha\circ (\mathbbm 1^3T^3)"]	&&\mathbb I^{\otimes 2}\arrow[rr,equal]&& \mathbb I
 	\end{tikzcd}
 \end{center} 
commutes, where $\shuffle_1 = (2,4)(3,5)$ and $\shuffle_2 = (2,4)(3,5)(6,8)(3,7)$, and we have shortened notation by omitting tensor product of morphisms, which has been indicated as juxtaposition. We also assume that $2$-cocycles are normalized, in the sense that $\alpha \circ (\eta \otimes \mathbbm 1^{\otimes 2}) = \alpha \circ (\mathbbm 1\otimes \eta \otimes \mathbbm 1) = \alpha \circ (\mathbbm 1^{\otimes 2}\otimes \eta) = \epsilon\otimes \epsilon$. 
}
\end{definition}
We now illustrate the prevoius definitions in the category of vector spaces, using Sweedler notation to denote comultiplication. We will give examples of the above structures later in the section, while we just assume in the following example that such objects exist in the category of vector spaces. 
\begin{example}\label{ex:vector2cocy}
	{\rm 
		Suppose that $\mathcal C$ is the category of vector spaces over some ground field $\mathbbm k$, and $X$ is as above. We want to show how the commutativity of categorical $2$-cocycle diagram translates in $\mathcal C$. Using Sweedler notation, on a basis vector $x\otimes y\otimes z\otimes u\otimes v$ we  get
		\begin{eqnarray*}
			&&	\alpha(x^{(1)}\otimes y^{(1)}\otimes z^{(1)})\cdot\alpha(T(x^{(2)}\otimes y^{(2)}\otimes z^{(2)})\otimes u\otimes v) \\
			&=&\alpha(x^{(1)}\otimes u^{(1)}\otimes v^{(1)})\\
			&&\hspace{1cm} \cdot \alpha(T(x^{(2)}\otimes u^{(2)}\otimes v^{(2)})\otimes T(y\otimes u^{(3)}\otimes v^{(3)})\otimes T(z\otimes u^{(4)}\otimes v^{(4)}))
		\end{eqnarray*}	
		where $\cdot$ indicates multiplication in $\mathbbm k$. 
	}
\end{example}

The fact that $\alpha$ is convolution invertible in the ``coefficients'' object $\mathbb I$ plays a fundamental role in constructing inverses in the general construction, as it will be seen below.

Lastly, we define an equivalence relation between categorical $2$-cocycles. 
\begin{definition}
	{\rm 
Let $\alpha$ and $\beta$ denote two categorical $2$-cocycles. We say that $\alpha$ and $\beta$ are equivalent if there exists a convolution invertible morphism $f: X\longrightarrow \mathbb I$ such that the following diagram 
\begin{center}
	\begin{tikzcd}
		X\otimes X\otimes X\arrow[rr,"\shuffle \circ \Delta^{\otimes 2}"] \arrow[d,swap,"(234)\circ (\Delta\otimes \mathbbm 1^{\otimes 2})"]& & X^{\otimes 3}\otimes X^{\otimes 3}\arrow[d,"\alpha\otimes (f\circ T)"]\\
		X^{\otimes 3}\otimes X\arrow[rr,"\beta\otimes f"] & & \mathbb I
		\end{tikzcd}
\end{center}	
commutes.
}
\end{definition}

We show the previous definition in the category of vector spaces.

\begin{example}
	{\rm 
Let $(X,T)$ be a TSD object in the category of vector spaces. Then two categorical $2$-cocycles $\alpha$ and $\beta$ are equivalent, by definition, if it holds $\alpha(x^{(1)}\otimes y^{(1)}\otimes z^{(1)})f(T(x^{(2)}\otimes y^{(2)}\otimes z^{(2)}))	= \beta(x^{(1)}\otimes y\otimes z)f(x^{(2)})$, for some $f$, for all $x,y,z\in X$. 
}
\end{example}
\subsection{Examples of categorical 2-cocycles}
We still need to provide examples of categorical $2$-cocycles, as per Definition~\ref{def:catcocy}. We begin by observing that the setting of linearized operations and set-theoretic ternary $2$-cocycles of Section~\ref{sec:ribboncat} provides first examples of such morphisms. 

\begin{example}\label{ex:groupcocy}
	{\rm 
	Let $Q$ be a ternary (set-theoretic) quandle and let $G$ be a (multiplicative) group. Suppose that $\alpha$ is a $2$-cocycle, i.e. $\alpha\in Z^2(Q,G)$. As in Section~\ref{sec:ribboncat} we let $X$ denote the linear space generated by the elements of $Q$ and define $T$ to be the linearized operation defined from the set-theoretic one of $Q$. Let $\chi: G\longrightarrow \mathbbm k^{\times}$ denote a group character. We define a linear map $\tilde \alpha : X\otimes X\otimes X \longrightarrow \mathbbm k$ by the assignment on simple vectors $x\otimes y\otimes z\mapsto \chi\alpha(x,y,z)$. Let us verify that $\tilde \alpha$ is indeed a categorical $2$-cocycle. Since the diagonal in $X$ is induced by the set-theoretic diagonal $x\mapsto x\times x$, applying Example~\ref{ex:vector2cocy} it is enough to verify
			\begin{eqnarray*}
			&&	\tilde \alpha(x\otimes y\otimes z)\cdot\tilde \alpha(T(x\otimes y\otimes z)\otimes u\otimes v) \\
			&=&\tilde \alpha(x\otimes u\otimes v)
			 \cdot \tilde \alpha(T(x\otimes u\otimes v)\otimes T(y\otimes u\otimes v)\otimes T(z\otimes u\otimes v))
		\end{eqnarray*}	
	which holds true using the definition of $\tilde \alpha$ and $T$ as linearizations of set-theoretic structures. 
}
\end{example}

We consider now, temporarily, the binary version of the previous constructions. Let us set $\mathcal C$ to be the category of vector spaces over a ground field $\mathbbm k$. Recall, in this case, that a unital coalgebra $(X,\Delta)$ of $\mathcal C$ along with a morphism of coalgebras $q : X\otimes X\longrightarrow X$ is said to be (binary) self-distributive if the equality
$$
q(q(x\otimes y)\otimes z) = q(q(x\otimes z^{(1)})\otimes q(y\otimes z^{(2)})),
$$
holds for all $x,y,z\in X$. The (binary) categorical $2$-cocycle condition for a convolution invertible morphism $\sigma: X\otimes X\longrightarrow \mathbbm k$ is readily introduced as follows
$$
\alpha(x^{(1)}\otimes y^{(1)}) \alpha(q(x^{(2)}\otimes y^{(2)})\otimes z) = \alpha(x^{(1)}\otimes z^{(1)}) \alpha(q(x^{(2)}\otimes z^{(2)})\otimes q(y\otimes z^{(3)})).
$$ 
As previously observed, in \cite{ESZ} Section~8, it has been proven that composing binary self-distributive operations yields TSD operations. In the set-theoretic case, moreover, it is shown that binary $2$-cocycles can be used to construct ternary $2$-cocycles. We want to use a vector space version of the set-theoretic result to obtain examples of (ternary) categorical $2$-cocycles. We consider the cocommutative case for simplicity. 
\begin{lemma}\label{lem:cocyclecomposition}
	In the setting above, suppose that $\alpha$ is a binary $2$-cocycle. Then, defining $\psi(x\otimes y\otimes z) := \alpha(x^{(1)}\otimes y^{(1)})\alpha(q(x^{(2)}\otimes y^{(2)})\otimes z)$, it follows that $\psi$ is a ternary $2$-cocycle for the doubled TSD operation $T= q\circ (q\otimes \mathbbm 1)$.
\end{lemma}
\begin{proof}
	The proof of this fact follows the same lines of Remark~5.7 in \cite{ESZ}, where the cocycles introduced in Remark~5.6 are assumed to coincide, and the binary operations are taken both to be $q$. The geometric intepretation of the various applications of the $2$-cocycle conditions are depicted as in Figure~3 in \cite{ESZ}, where extra care is to be taken to utilize the correct superscript corresponding to Sweedler notation for comultiplication. To prove that $\psi$ is convolution invertible, one needs the fact that $q$ is a morphism of coalgebra, which is assumed by hypothesis. 
\end{proof}
\begin{example}\label{ex:Hopfcocy}
	{\rm 
		Let $X$ be an involutory Hopf algebra over a ground field $\mathbbm k$, and let $\alpha$ denote a $2$-cocycle with coefficients in $\mathbbm k$. Recall, from the introduction of \cite{AS}, or \cite{KSq} Section~10.2.3, that this means that $\sigma: X\otimes X\longrightarrow \mathbbm k$ is convolution invertible, satisfies the equation
		$$
		\sigma(x^{(1)}\otimes y^{(1)})\sigma(x^{(2)}y^{(2)}\otimes z) = \sigma(x\otimes y^{(1)}z^{(1)})\sigma(y^{(2)}\otimes z^{(2)}),
		$$
		for all $x,y,z\in X$, and it is normalized $\sigma(1\otimes x) = \sigma(x\otimes 1) = \epsilon(x)$. Let us consider now the (double) quantum conjugation operation of Example~\ref{ex:qconjugation}. Then, assuming that the underlying coalgebra structure is cocommutative for simplicity, one can verify that the morphism $\alpha: X\otimes X\longrightarrow \mathbbm k$ defined as $\alpha(x\otimes y) := \sigma(x^{(1)}\otimes y^{(1)})\sigma^{-1}(y^{(2)}\otimes S(y^{(3)})x^{(2)}y^{(4)})$ satisfies the (binary) self-distributive cocycle condition. Applying Lemma~\ref{lem:cocyclecomposition} we obtain a (ternary) categorical $2$-cocycle. In fact, the correspondence $\sigma \mapsto \alpha$ is a quantum version of the map given in \cite{CJKLS}, Theorem~7.1. This was originally motivated by diagrammatic computations. 
}
\end{example}

Cocycles of Hopf algebras are used to twist the product structure and the antipode, to obtain a new Hopf algebra \cite{AS,KSq}. It is known, and easily verified, that when the underlying coalgebra structure is cocommutative, Hopf $2$-cocycles generate the same product of the Hopf algebra one starts with. Example~\ref{ex:Hopfcocy} considers, in fact,  a subclass of a family of $2$-cocycles that are referred to as {\it invariant cocycles}, or {\it lazy cocycles} \cite{EG}. These are defined as those $2$-cocycles whose corresponding twisted Hopf algebra structure is unchanged. Invariant cocycles are defined up to an equivalence relation, and their equivalence classes constitute a group, as for the cohomology of groups, or algebras etc. It seems a relevant question to study equivalence classes of categorical $2$-cocycles obtained from invariant cocycles. Moreover, as seen in the next section, categorical $2$-cocycles are used to twist the Yang-Baxter operator associated to a TSD object, and the ribbon category constructed below is well defined within an equivalence class. It might be of interest to relate the equivalence classes of invariant $2$-cocycles to those of ribbon categories associated to TSD cocycles. We point out that we do not know whether it is possible to construct categorical TSD $2$-cocycles from Hopf $2$-cocycles that are not invariant.

\subsection{Construction of ribbon categories} 
We define a general version of ribbon categories from TSD objects in symmetric monoidal categories. This construction gives the linearized definition in Section~\ref{sec:ribboncat} when the symmetric monoidal category is that of vector spaces over a field, the TSD object is defined by linearizing a set-theoretic ternary rack/quandle, and the coefficients are taken to be a subgroup of the units of the ground field, via a group character. 

Let $\mathcal C$ be an $\mathbb I$-linear symmetric monoidal category, and let $X$ be a ternary self-distributive object in $\mathcal C$. Let $\alpha$ be a categorical $2$-cocycle in the sense of Definition~\ref{def:catcocy}. Then we define a braiding morphism $c_{2,2}^\alpha$ following the case of linearized set-theoretic operations as follows
$$
c_{2,2}^\alpha = (\mathbbm 1^\otimes 2\otimes ([\alpha\otimes \alpha]\otimes  T\otimes T))\shuffle_c(\Delta^{\otimes 2}\Delta_4^{\otimes 2}),
$$
where $\shuffle_c$ is the morphism corresponding to the permutation  
$$\sigma = \bigl(\begin{smallmatrix}
1 & 2 & 3 & 4& 5 &6 & 7 & 8 & 9 & 10 & 11 & 12 & 13 & 14 \\
3 & 9 & 6 & 12 & 1 & 4 & 7 & 10 & 13 & 2 & 5 & 8 & 11 & 14
\end{smallmatrix}\bigr).
$$
Similarly we define a twisting morphism as follows
$$
\theta^\alpha_2 =([\alpha\otimes \alpha]\otimes T\otimes T) \shuffle_\theta(\Delta_6^{\otimes 2}),
$$
where $\shuffle_\theta$ is the morphism corrsponding to the permutation 
$$
\sigma = \bigl(\begin{smallmatrix}
1 & 2 & 3 & 4& 5 &6 & 7 & 8 & 9 & 10 & 11 & 12 \\
1& 2 & 5 & 7 & 8  & 11 & 4 & 3 & 6 & 10 & 9 & 12
\end{smallmatrix}\bigr).
$$
\begin{example}
	{\rm 
We want to illustrate braiding and twisting morphisms, as given above, in the category of vector spaces. Let $\alpha$ be a categorical $2$-cocycle as in Example~\ref{ex:vector2cocy}. We now explicitly give the form of switching and twisting morphisms on simple tensors. We have for $c_{2,2}^\alpha$
\begin{eqnarray*}
	c_{2,2}^\alpha(x\otimes y \otimes \otimes z\otimes w) &=& z^{(1)}\otimes w^{(1)} \otimes \\
	&&[\alpha(x^{(1)}\otimes z^{(2)}\otimes w^{(2)})\cdot \alpha(y^{(1)}\otimes z^{(3)}\otimes w^{(3)})]\cdot \\
	&& T(x^{(2)}\otimes z^{(4)}\otimes w^{(4)})\otimes T(y^{(2)}\otimes z^{(5)}\otimes w^{(5)}),
\end{eqnarray*}
and for twisting morphism $\theta^\alpha_2$
\begin{eqnarray*}
	\theta^\alpha_2 (x\otimes y) &=& [\alpha(x^{(1)}\otimes x^{(2)}\otimes y^{(2)})\cdot \alpha(y^{(1)}\otimes x^{(3)}\otimes y^{(3)})]\cdot \\
	&& T(x^{(4)}\otimes x^{(5)}\otimes y^{(5)})\otimes T(y^{(4)}\otimes x^{(6)}\otimes y^{(6)}).
\end{eqnarray*}	
}
\end{example}

We suppose now, for simplicity, that the comultiplication of $X$ is cocommutative. Although, in principle this condition can be weakened to some classes of object whose comultiplication satisfies some special symmetries in the next results, the proofs are significantly easier when dealing with a cocommutative object. For instance, the comultiplication of Example~\ref{ex:klinear} is not cocommutative. It is possible to show nontheless that the following constructions hold. We consider this class of objects in a subsequent article.

\begin{definition}
	{\rm 
	Let $\mathcal C$ be an $\mathbb I$-linear symmetric monoidal category with duals, $(X,T)$ a TSD object in $\mathcal C$, and $\alpha$ a categorical $2$-cocycle of $X$ with coefficients in $\mathbb I$. Let $c_{2,2}^\alpha$ and $\theta_2^\alpha$ be braiding a twisting morphisms, respectively, as previously defined in this section. Then we define $\mathcal R^*_\alpha(X)$ to be the monoidal category freely generated by $X\otimes X$ where morphisms are defined as in Definition~\ref{def:ribboncat} from $c_{2,2}^\alpha$ and $\theta_2^\alpha$. 
}
\end{definition}

\begin{theorem}\label{thm:ribbongeneral}
$\mathcal R^*_\alpha(X)$ is a ribbon category. Moreover, if $\alpha$ and $\beta$ are equivalent, then $\mathcal R^*_\alpha(X) \cong \mathcal R^*_\beta(X)$ as ribbon categories.
\end{theorem}
Before proceeding with the actual proof, we give some key steps of the proof in the setting of the category of vector spaces. This special case is rather illuminating when considering the actual proof, and functions as prototype for the general case. We want to show that the braid equation holds, by computing the LHS and RHS on a simple tensor $x\otimes y\boxtimes z\otimes w\boxtimes u\otimes v\in X^{\otimes 2}\boxtimes X^{\otimes 2}\boxtimes X^{\otimes 2}$. We obtain, using Sweedler's notation, and applying one of the maps whose composition is the LHS of the braid equation
\begin{eqnarray*}
\lefteqn{x\otimes y\boxtimes z\otimes w\boxtimes u\otimes v}\\
&\mapsto &z^{(1)}\otimes w^{(1)}\boxtimes\\
&&\hspace{0.5cm} [\alpha(x^{(1)}\otimes z^{(2)}\otimes w^{(2)})\alpha(y^{(1)}\otimes z^{(3)}\otimes w^{(3)})]\cdot\\
&&\hspace{0.5cm} T(x^{(2)}\otimes z^{(4)}\otimes w^{(4)})\otimes T(y^{(2)}\otimes z^{(5)}\otimes w^{(5)})\boxtimes u\otimes v\\
&\mapsto& z^{(1)}\otimes w^{(1)}\boxtimes u^{(1)}\otimes v^{(1)}\boxtimes [\alpha(T(x^{(21)}\otimes z^{(41)})\otimes w^{(41)}\otimes u^{(2)}\otimes v^{(2)})\\
&&\hspace{0.5cm} \alpha(T(y^{(21)}\otimes z^{(51)})\otimes w^{(51)}\otimes u^{(3)}\otimes v^{(3)})\alpha(x^{(1)}\otimes z^{(2)}\otimes w^{(2)})\\
&&\hspace{0.5cm} \alpha(y^{(1)}\otimes z^{(3)}\otimes w^{(3)})]\cdot T(T(x^{(22)}\otimes z^{(42)}\otimes w^{(42)}\otimes u^{(4)}\otimes v^{(4)})\\
&& \hspace{0.5cm} \otimes T(T(y^{(22)}\otimes z^{(52)}\otimes w^{(52)}\otimes u^{(5)}\otimes v^{(5)})\\
&\mapsto& u^{(11)}\otimes v^{(11)}\boxtimes [\alpha(z^{(11)}\otimes u^{(12)}\otimes v^{(12)})\alpha(w^{(11)}\otimes u^{(13)}\otimes v^{(13)})]\\
&& \hspace{0.5cm} \cdot T(z^{(12)}\otimes u^{(14)}\otimes v^{(14)})\otimes T(w^{(12)}\otimes u^{(15)}\otimes v^{(15)})\boxtimes \\
&& \hspace{0.5cm} [\alpha(T(x^{(21)}\otimes z^{(41)})\otimes w^{(41)}\otimes u^{(2)}\otimes v^{(2)})\\
&&\hspace{0.5cm} \alpha(T(y^{(21)}\otimes z^{(51)})\otimes w^{(51)}\otimes u^{(3)}\otimes v^{(3)})\alpha(x^{(1)}\otimes z^{(2)}\otimes w^{(2)})\\
&&\hspace{0.5cm} \alpha(y^{(1)}\otimes z^{(3)}\otimes w^{(3)})]\cdot T(T(x^{(2)}\otimes z^{(42)}\otimes w^{(42)}\otimes u^{(4)}\otimes v^{(4)})\\
&& \hspace{0.5cm} \otimes T(T(y^{(22)}\otimes z^{(52)}\otimes w^{(52)}\otimes u^{(5)}\otimes v^{(5)}).
\end{eqnarray*}
 Let us now compute the RHS of the braid equation when evaluated on $x\otimes y\otimes z\otimes w\otimes u\otimes v$. We do not write each step, but rather provide the final result
\begin{eqnarray*}
\lefteqn{(\mathbbm 1\boxtimes c_{2,2}^\alpha)\circ(c_{2,2}^\alpha\boxtimes \mathbbm 1)\circ (\mathbbm 1\boxtimes c_{2,2}^\alpha)(x\otimes y\boxtimes z\otimes w\boxtimes u\otimes v)}\\
&=& u^{(11)}\otimes v^{(11)}\boxtimes [\alpha(z^{(1)}\otimes u^{(2)}\otimes v^{(2)})\alpha(w^{(1)}\otimes u^{(3)}\otimes v^{(3)})]\\
&& \hspace{0.5cm}\cdot T(z^{(21)}\otimes u^{(41)}\otimes v^{(41)})\otimes T(w^{(21)}\otimes u^{(51)}\otimes v^{(51)})\boxtimes \\
&& \hspace{0.5cm} [\alpha(T(x^{(21)}\otimes u^{(141)}\otimes v^{(141)})\otimes T(z^{(22)}\otimes u^{(42)}\otimes v^{(42)})\otimes T(w^{(22)}\otimes\\
&& \hspace{2cm} u^{(52)}\otimes v^{(52)}) )\\
&&\hspace{0.5cm} [\alpha(T(y^{(21)}\otimes u^{(151)}\otimes v^{(151)})\otimes T(z^{(23)}\otimes u^{(43)}\otimes v^{(43)})\otimes T(w^{(23)}\otimes\\
&& \hspace{2cm} u^{(53)}\otimes v^{(53)}))\alpha(x^{(1)}\otimes u^{(12)}\otimes v^{(12)})\alpha(y^{(1)}\otimes u^{(13)}\otimes v^{(13)})]\\
&& \hspace{0.5cm} \cdot T(T(x^{(12)}\otimes u^{(142)}\otimes v^{(142)})\otimes T(z^{(24)}\otimes u^{(44)}\otimes v^{(44)})\otimes T(w^{(24)}\otimes\\
&& \hspace{2cm} u^{(54)}\otimes v^{(54)}))\otimes \\
&&\hspace{0.5cm} T(T(y^{(12)}\otimes u^{(152)}\otimes v^{(152)})\otimes T(z^{(15)}\otimes u^{(45)}\otimes v^{(45)})\otimes T(w^{(25)}\otimes\\
&& \hspace{2cm} u^{(55)}\otimes v^{(55)}))).
\end{eqnarray*}
Now, comparing the two expressions we see that we cannot apply directly ternary self-distributivity of $T$ and the categorical $2$-cocycle condition to $\alpha$, as the terms corresponding to comultiplications in Sweedler notation are shuffled differently. We can rearrenge them by means of the cocommutativity of $X$ and then conclude that they coincide applying categorical $2$-cocycle condition and TSD property of $T$. 

A similar direct reasoning is applied to show that braiding morphism and twisting commute. We will not add the details of this computation.

The proof of Theorem \ref{thm:ribbongeneral} will consist of a similar reasoning, but replacing equations by commutative diagrams. We compare the expressions by writing them as compositions of ``all comultiplications'', ``shuffle'', ``$T$'s and $\alpha$'s'' and finally multiplying the $\alpha$'s by identifying two copies of the unit object $\mathbb I$ with $\mathbb I$ itself. Then using cocommutativity and naturality of shuffle morphisms in a symmetric monoidal category, one draws the conclusion that the two terms are equal. 

\begin{proof}[Proof of Theorem \ref{thm:ribbongeneral}]
	As in the proof of Theorem~\ref{thm:ribboncat}, the crucial step is to show that $c_{2,2}^\alpha$ satisfies the braid equation, and that $\theta_2^\alpha$ satisfies the compatibility relation with respect to $c_{2,2}^\alpha$. We proceed to prove the first assertion. We need to show that the diagram
	
	\begin{center}
		\begin{tikzcd}
			&X^{\otimes 2}\boxtimes X^{\otimes 2}\boxtimes X^{\otimes 2}\arrow[dl,swap,"c_{2,2}^\alpha\boxtimes \mathbbm 1^{\otimes 2}"]\arrow[dr,"1^{\otimes 2}\boxtimes c_{2,2}^\alpha"]&\\
			X^{\otimes 2}\boxtimes X^{\otimes 2}\boxtimes X^{\otimes 2}\arrow[dd,swap,"1^{\otimes 2}\boxtimes c_{2,2}^\alpha"] & & X^{\otimes 2}\boxtimes X^{\otimes 2}\boxtimes X^{\otimes 2}\arrow[dd,"c_{2,2}^\alpha\boxtimes \mathbbm 1^{\otimes 2}"]\\
			& &\\
			X^{\otimes 2}\boxtimes X^{\otimes 2}\boxtimes X^{\otimes 2}\arrow[dr,swap,"c_{2,2}^\alpha\boxtimes \mathbbm 1^{\otimes 2}"]& & X^{\otimes 2}\boxtimes X^{\otimes 2}\boxtimes X^{\otimes 2}\arrow[dl,"1^{\otimes 2}\boxtimes c_{2,2}^\alpha"]\\
			&X^{\otimes 2}\boxtimes X^{\otimes 2}\boxtimes X^{\otimes 2}&
			\end{tikzcd}
	\end{center}
is commutative. We will refer to the left perimeter, from top to bottom, of the preceding hexagon as the ``left-hand side'' of the braid diagram, or simply LHS, and similarly for the right perimeter we will say the ``right-hand side'' or RHS of the braid diagram. In what follows we will keep using the same shortened notation in which a tensor product of maps of such as $\mathbbm 1^{\otimes 2}\boxtimes T^{\otimes 2}$ will be denoted by juxtaposition, $\mathbbm 1^{2}T^2$, where powers of tensor products of type $\otimes 2$ are simply indicated by a power $2$. A similar notation is also used for objects, where we omit the tensor symbol $\otimes$ in the exponents, although we keep the $\boxtimes$ symbol for clarity.

The proof will consists of rewriting the LHS and the RHS of the braid diagram in a convenient way, by using the axioms of the symmetric monoida category $\mathcal C$, and the left Frobenius module axiom. Then we will argue that the two expressions coincide by means of the categorical $2$-cocycle condition, and the rack axioms of $T$. We start with the LHS.

 Using the definition of $c_{2,2}^\alpha$, the LHS of the braid diagram fits in the commutative diagram of Figure~\ref{fig:diagram1}.
 \begin{figure}
	\begin{center}
		\begin{tikzcd}
			X^{ 2}\boxtimes X^{ 2}\boxtimes X^{ 2}\arrow[dddddddr,bend right,swap,"(c_{2,2}^\alpha \mathbbm 1^{ 2})\circ(\mathbbm 1^{ 2} c_{2,2}^\alpha)\circ(c_{2,2}^\alpha \mathbbm 1^{2})"]\arrow[dddddr,bend right,"c_{2,2} \mathbbm{1^{ 2}}"] \arrow[dr,"\Delta^{ 2} \Delta_4^{ 2} \mathbbm{1^{2}}"]& \\
			& X^{ 4}\boxtimes X^{ 10}\boxtimes X^{ 2} \cong X^{ 2}\boxtimes X^{ 12}\boxtimes X^{ 2}\arrow[d,"\shuffle_c 1^{ 2}"]\\
			&  X^{ 2}\boxtimes X^{ 12}\boxtimes X^{ 2}\arrow[d,"\alpha^2  T^2 \mathbbm 1^{ 2}"]\\
			& X^{2}\boxtimes  X^{2}\otimes X^{ 2}\boxtimes X^{ 2}\arrow[d,"\mu \mathbbm1^{2} \mathbbm 1^{2}"]\\
			& X^{ 2}\boxtimes X^{ 3}\boxtimes X^{ 2}\arrow[d,"\mathbbm 1^{ 2} \mathbbm 1^{2}"]\\
			& X^{ 2}\boxtimes X^{2}\boxtimes X^{2}\arrow[d,"\mathbbm 1^{ 2} c_{2,2}^\alpha"]\\
			& X^{ 2}\boxtimes X^{2}\boxtimes X^{2}\arrow[d,"c_{2,2}^\alpha\mathbbm 1^{2}"]\\
			& X^{2}\boxtimes X^{2}\boxtimes X^{2}
		\end{tikzcd}
	\end{center}
\caption{}
\label{fig:diagram1}
\end{figure}
Now we rewrite the top-right perimeter of the previous diagram as in Figure~\ref{fig:diagram2},
\begin{figure}
\begin{center}
	\begin{tikzcd}
		X^{2}\boxtimes X^{2}\boxtimes X^{2}\arrow[d,swap,"\Delta^{ 2} \Delta_4^{ 2} \mathbbm 1^{2}"]\arrow[r,"\Delta^2\Delta_4^2\Delta_4^2"]&X^{4}\boxtimes X^{ 10}\boxtimes X^{10}\arrow[d,equal]\arrow[r,"(\Delta\mathbbm 1)^2(\Delta\mathbbm 1^4)^2\mathbbm 1^{10}"]& X^2\boxtimes X^{12}\boxtimes X^{10}\arrow[dd,"\shuffle"]\\
		X^{ 2}\boxtimes X^{ 12}\boxtimes X^{2}\arrow[r,"\mathbbm 1^{14}\Delta_4^2"]\arrow[d,swap,"\shuffle_c \mathbbm 1^{ 2}"]&X^2\boxtimes X^{12}\boxtimes X^{10}\arrow[d,"(\Delta\mathbbm 1)^2(\Delta\mathbbm 1^4)^2\mathbbm 1^{10}"]&\\
		X^{ 2}\boxtimes  X^{12}\boxtimes X^{ 2}\arrow[r,"\mathbbm 1^8\Delta^6\Delta_4^2"]\arrow[d,swap,"\mathbbm 1^2\alpha^2T^2\mathbbm 1^2"]&X^2\boxtimes X^6\otimes X^{12}\boxtimes X^{10}\arrow[d,"\mathbbm 1^2\alpha^2T^4\mathbbm 1^{10}"]\arrow[r,"\shuffle"]& X^2\boxtimes X^2\boxtimes X^{26}\arrow[d,swap,"\mathbbm 1^4(T\mathbbm 1^2)^2\alpha^2(T\mathbbm 1^2)^2"]\\
		X^{ 2}\boxtimes \mathbb I^2\otimes X^{ 2}\boxtimes X^{ 2}\arrow[r,"\mathbbm 1^4\Delta^2\Delta_4^2"]\arrow[d,swap,"\mathbbm 1^2\mu\mathbbm 1^4"]&X^2\boxtimes \mathbb I^2\otimes X^4\boxtimes X^{10}\arrow[d,"\mathbbm 1^2 (\mu\mathbbm 1^4)\mathbbm 1^{10}"]\arrow[r,"\shuffle"]&X^2\boxtimes X^2\boxtimes X^6\otimes \mathbb I^2\otimes X^6\arrow[d,swap,"\mathbbm 1^{10}\mu\mathbbm 1^6"]\\
		X^{ 2}\boxtimes \mathbb I\otimes X^{ 2}\boxtimes X^{ 2}\arrow[r,"(\mathbbm 1^3\Delta^2\Delta_4^2)"]\arrow[d,swap,"\mathbbm 1^2\mathbbm 1^3"]&X^2\boxtimes X^4\boxtimes X^{10}\arrow[r,"\mathbbm 1^2\shuffle_c"]& X^2\boxtimes X^2\boxtimes X^{12}\arrow[d,swap,"(\mathbbm 1^4(\mu\mathbbm 1)\mathbbm 1^2)\circ(\mathbbm 1^{10}T^2)"]\\
		X^{ 2}\boxtimes X^{ 2}\boxtimes X^{ 2}\arrow[ur,swap,"\mathbbm 1^2\Delta^2\Delta_4^2"]\arrow[rr,"\mathbbm 1^2c_{2,2}^\alpha"]\arrow[rrd,bend right,"(c_{2,2}^\alpha\mathbbm 1^2)\circ(\mathbbm 1^2c_{2,2}^\alpha)"]&&X^{ 2}\boxtimes X^{ 2}\boxtimes X^{ 2}\arrow[d,"c_{2,2}^\alpha\mathbbm 1^2"]\\
		& & X^{ 2}\boxtimes X^{ 2}\boxtimes X^{ 2}
		\end{tikzcd}
\end{center}
\caption{}
\label{fig:diagram2}
\end{figure}
where we have used naturality of the shuffle morphisms in the category $\mathcal C$, coassociativity of the morphism $\Delta$ and, in the bottom-left triangle, the fact that $\mathbb I$ commutes with objects  of $\mathcal C$. Similar considerations show that the diagram of Figure~\ref{fig:diagram3}
\begin{figure}
\begin{center}
	\begin{tikzcd}
		X^2\boxtimes X^{12}\boxtimes X^{10}\arrow[dd,swap,"\shuffle"]\arrow[rr,bend left]\arrow[r,"\mathbbm 1^6(\Delta_2\mathbbm 1^5)^2(\Delta_4\mathbbm 1^4)^2"] & X^2\boxtimes X^{20}\boxtimes X^{18}\arrow[dd,"\shuffle"]& X^2\boxtimes X^2\boxtimes X^2\\
		& &  \\
		X^2\boxtimes X^2\boxtimes X^{26}\arrow[d,swap,"\mathbbm 1^4(T\mathbbm 1)^2\alpha^2(T\mathbbm 1^2)^2"]\arrow[r,""]&X^2\boxtimes X^{12}\boxtimes X^{26}\arrow[d]\arrow[uur]& \\
		X^2\boxtimes X^2\boxtimes X^6\otimes \mathbb I^2\otimes X^6\arrow[d,swap,"\mathbbm 1^4(\mu\circ\mu^2)\mathbbm 1^8"]\arrow[r]&X^4\boxtimes X^{10}\boxtimes X^6\otimes \mathbb I^2\otimes X^6&\\
		X^2\boxtimes X^2\boxtimes X^{12}\arrow[d,swap,"\mathbbm 1^4\alpha^2T^2"]& &\\
		X^2\boxtimes X^2\boxtimes \mathbb I^2\otimes X^2\arrow[d,swap,"\mathbbm 1^4\mu\mathbbm 1^2"]\arrow[uur]& &\\
		X^2\boxtimes X^2\boxtimes \mathbb I\otimes X^2\arrow[d,swap,"\mathbbm 1^4\mathbbm1"]& &\\
		X^2\boxtimes X^2\boxtimes X^2\arrow[uuuur, bend right=15,swap,"\Delta^2\Delta_4^2\mathbbm 1^2"]\arrow[uuuuuuurr,bend right,swap,"c_{2,2}^\alpha\mathbbm 1^2"] & &
		\end{tikzcd}
	\end{center}
\caption{}
\label{fig:diagram3}
\end{figure}
commutes, where the horizontal maps are obtained from (appropriate) products of the comultiplication morphism, and similarly for the hypotenuse of the central triangle.  The morphism $X^2\boxtimes X^{12} \boxtimes X^{26} \longrightarrow X^2\boxtimes X^2\boxtimes X^2$ is $\mathbbm 1^2 \alpha^2T^2((\mu\circ\mu^2)\circ (T\mathbbm 1^2)^2)T\circ (T\mathbbm 1^2)^2$. Pasting together the three diagrams along the corresponding edges shows that the morphism $(c_{2,2}^\alpha\boxtimes \mathbbm 1)\circ (\mathbbm 1\boxtimes c_{2,2}^\alpha)\circ (c_{2,2}^\alpha\boxtimes \mathbbm 1)$ can be rewritten by first aplying the comultiplication morphisms $\Delta$ all together to produce the right number of copies of $X$, shuffle them in the right position and then apply all the morphisms $T$ and $\alpha$ at the end. This is in fact better understood by thinking of the same proof for Hopf algebras.  A (very tedious) direct inspection, simplified by diagrammatic reasoning, can be applied to determine step by step the shuffles used in the commutative diagrams. Thus proceeding we see that we have obtained for the LHS of the braid diagram $(c_{2,2}^\alpha\boxtimes \mathbbm 1)\circ (\mathbbm 1\boxtimes c_{2,2}^\alpha)\circ (c_{2,2}^\alpha\boxtimes \mathbbm 1) = (\mathbbm 1^{\otimes 2}\boxtimes (\mathbbm 1\circ\mu\otimes \mathbbm 1^{\otimes 2})\boxtimes (\mu^{\otimes }))\circ(\mathbbm 1^2\boxtimes \alpha^{\otimes 2}T^{\otimes 2}\boxtimes (\alpha\circ (T\otimes \mathbbm 1^{\otimes 2}))^{\otimes 2}\otimes \alpha^{\otimes 2}\otimes (T\circ (T\otimes \mathbbm 1^{\otimes 2})^{\otimes 2}))\circ\shuffle\circ(\Delta_2^{\otimes 2}\boxtimes \Delta_7^{\otimes 2}\boxtimes \Delta_{8}^{\otimes 2})$, where the shuffle morphism $\shuffle$ corresponds to the permutation that rearrenges the terms in the last expression of the computation of $(c_{2,2}^\alpha\boxtimes \mathbbm 1)\circ (\mathbbm 1\boxtimes c_{2,2}^\alpha)\circ (c_{2,2}^\alpha\boxtimes \mathbbm 1)$. A similar procedure shows that the RHS of the braid diagram is written as $ (\mathbbm 1\boxtimes c_{2,2}^\alpha)\circ (c_{2,2}^\alpha\boxtimes \mathbbm 1)\circ  (\mathbbm 1\boxtimes c_{2,2}^\alpha) =  (\mathbbm 1^{\otimes 2}\boxtimes (\mu\circ\alpha^{\otimes 2})\circ T^{\otimes 2})\boxtimes (\mu\circ(\mu\circ \alpha^{\otimes 2}\otimes \mu\circ \alpha^{\otimes 2})\otimes (T\circ T^{\otimes 3})^{\otimes 2})\shuffle\circ(\Delta_2^{\otimes 2}\boxtimes \Delta_5^{\otimes 2}\boxtimes \Delta_{18}^{\otimes 2})$, where the shuffle morphism can be seen to coincide with the morphism corresponding to the permutation giving rise to the reordering of elements given in the RHS of the computation preceding the present proof. 

To show that twist morphism $\theta^\alpha_2$ satisfies the compatibility condition with braiding morphism $c_{2,2}^\alpha$, we can proceed similarly. First we rewrite the two perimeters of the diagram 
\begin{center}
	\begin{tikzcd}
	&	X^{\otimes 2}\boxtimes X^{\otimes 2}\arrow[dl,swap,"\theta^\alpha_2\boxtimes \mathbbm 1"] \arrow[dr,"c_{2,2}^\alpha"]  &\\
	X^{\otimes 2}\boxtimes X^{\otimes 2}\arrow[dr,swap,"c_{2,2}^\alpha"] & & X^{\otimes 2}\boxtimes X^{\otimes 2}\arrow[dl,"\mathbbm 1\boxtimes \theta^\alpha_2"] \\
	& X^{\otimes 2}\boxtimes X^{\otimes 2} & 
	\end{tikzcd}
\end{center}
as composition of comultiplications in $X$ and the appropriate shuffle, then applying the morphisms $T$, $\alpha$ and multiplication $\mu$. Then we can compare the two expressions by using TSD property of $T$ and categorical $2$-cocycle condition of $\alpha$. Similar considerations are also applied to the diagram proving that $(\mathbbm 1\boxtimes \theta^\alpha_2)\circ c_{2,2}^\alpha = c_{2,2}^\alpha\circ(\theta^\alpha_2\boxtimes \mathbbm 1)$. 

We need to prove now that the morphism $c_{2,2}^\alpha$ has an inverse in the category $\mathcal C$. Upon defining the morphism $\hat c^\alpha_{2,2} := (\mu\otimes \mathbbm 1^{\otimes 2}\boxtimes \mathbbm 1^{\otimes 2})\circ(\tilde\alpha^{\otimes 2}\boxtimes \mathbbm 1^{\otimes 2})\circ [(T\otimes \mathbbm 1)^{\otimes 2}\otimes (T^{\otimes 2}\boxtimes \mathbbm 1^{\otimes 2})]\circ\shuffle\circ(\Delta_6^{\otimes 2}\boxtimes \Delta^{\otimes 2})$
where we have indicated $\tilde \alpha : X^{\otimes 3} \longrightarrow \mathbb I$ the convolution inverse of $\alpha$, and the shuffle morphism corresponds to the permutation
$$\sigma = \bigl(\begin{smallmatrix}
1 & 2 & 3 & 4 & 5 &6 & 7 & 8 & 9 & 10 & 11 & 12 & 13 & 14 & 15 & 16 & 17 & 18 \\
3 & 4 & 8 & 9 & 12 & 15 & 16 & 17 & 2 & 5 & 7 & 10 & 12 & 18 & 11 & 1 & 6 & 14
\end{smallmatrix}\bigr),
$$
we see that $\hat c_{2,2}^\alpha$ is a left and right inverse of $c_{2,2}^\alpha$, which gives the required invertibility. It is fundamental here, that $\alpha$ is convolution invertible. 

Once the fact that $c_{2,2}^\alpha$ satisfies the braid equation, the compatibility condition with $\theta_2^\alpha$ and invertibility, the rest of the proof proceeds as that of Theorem~\ref{thm:ribboncat}, as the category $\mathcal R^*(X)$ is defined inductively as in Section~\ref{sec:ribboncat}.

To show that, given equivalent $2$-cocycles $\alpha$ and $\beta$, there exists an equivalence of ribbon categories $\mathcal R^*_\alpha(X) \cong \mathcal R^*_\beta(X)$ one proceeds as in the counterpart of the same proof of Theorem~\ref{thm:ribboncat}, where we use that a morphism $f:X\longrightarrow \mathbb I$ giving the equivalence between $\alpha$ and $\beta$ is convolution invertible. 
\end{proof}

Since over an algebraically closed fields of characteristic zero $\mathbbm k$, a cocommutative finite dimensional Hopf algebra $X$ is isomorphic to the group algebra $\mathbbm k[G]$ for some group $G$, in a sense, the construction of Section~\ref{sec:ribboncat} is general enough for TSD structures obtained via heaps. The ribbon category of Theorem~\ref{thm:ribbongeneral} gives new results when applied to finite dimensional Hopf algebras over a ring/field that is not an algebraically closed field of characteristic zero. 

\begin{remark}
	{\rm 
Observe that Theorem~\ref{thm:ribbongeneral} generalizes the results of Section~\ref{sec:ribboncat} to the setting of cohomology coefficients in a non-commutative group $G$, since from a $2$-cocycle with coefficients in $G$ and a group character $\chi: G\longrightarrow \mathbbm k^\times$, we obtain a ribbon category in the category of vector spaces over $\mathbbm k$, where braiding and twisting are generated by $c^\alpha_{2,2}$ and $\theta^\alpha_2$. 
}
\end{remark}

\section{Invariants of framed links in symmetric monoidal categories}\label{sec:quantumgeneral} 

Lastly, we describe in this section how the construction of Section~\ref{sec:generalized}, whose twisting morphisms are revisited here, gives rise to quantum invariants of framed links in symmetric monoidal categories. This discussion follows standard arguments, already introduced in Section~\ref{sec:ribbonquantum} for instance. 

Recall that twists in a braided category, along with duals, uniquely determine a pivotal structure, i.e. an isomorphism of objects $X$ with their double duals $X^{**}$. The family of twists defined above, with the corresponding choice of pivotal structure can be used to define invariants of framed links. It is important to note, though, that duality of an object $X$ in a symmetric monoidal category can be regarded as a finiteness condition that ensures the existence of quantum trace. 

In fact, for an important class of TSD objects, namely that of involutory Hopf monoids in a symmetric monoidal category, there often exists another natural choice of twists, and therefore pivotal structure, that can be used to construct framed link invariants. Before explicitly giving the definition of link invariants from the category $\mathcal R^*_\alpha(X)$, we discuss how to obtain new twists for involutory Hopf algebras, and more generally Hopf monoids, under some extra conditions.

First recall that a Frobenius algebra $(X,\mu,\eta,\Delta,\epsilon)$ in a symmetric monoidal category $\mathcal C$ is a bialgebra such that the following Frobenius axiom holds
$$
(\mathbbm 1\otimes \mu)\circ (\Delta \otimes \mathbbm 1) = \Delta \circ \mu = (\mu \otimes \mathbbm 1)\circ (\mathbbm 1\otimes \Delta).
$$
Given a Frobenius algebra $X$, we can define an (associative) pairing and a coparining, indicated as $\cup: X\otimes X\longrightarrow \mathbb I$ and $\cap : \mathbb I\longrightarrow X\otimes X$ respectively, as $\cup = \epsilon \circ \mu$ and $\cap = \Delta \circ \eta$. Paring and coparing in a Frobenius algebra satisfy the axiom 
$$
(\mathbbm 1\otimes \cup) \circ (\cap \otimes \mathbbm 1) = \mathbbm 1 = (\cup\otimes \mathbbm 1) \circ (\mathbbm 1\otimes \cap),
$$
which determines that $\cup$ is an associative non-degenerate pairing. It in fact turns out that a monoid endowed with an associative pairing and a copairing satisfying the compatibility condition given above is equivalent to the previous definition of Frobenius algebra. 

Suppose we are given an involutory Hopf algebra which is finitely generated over a ground PID $\mathbbm k$. We know that a Frobenius algebra arises from such a Hopf algebra, from traditional results of Larson and Sweedler \cite{LS}. The Frobenius structure then gives an identification of $X$ with its dual $X^*$ and, consequently, a pivotal structure where $X \cong X^{**}$ is obtained through the Frobenius pairing. Let us indicate by $\cup$ and $\cap$ the Frobenius pairings. 

Let us now consider an $\mathbb I$-linear symmetric monoidal category $\mathcal C$, and let $X$ denote a (cocommutative) Hopf monoid in $\mathcal C$. We define a braiding $c$ using the quantum heap operation of $X$, as in the previous section. Suppose $X$ has an integral and a cointegral. These are, by definition, a point $\lambda : \mathbb I\longrightarrow X$ and a copoint $\gamma: X\longrightarrow \mathbb I$, respectively, such that 
\begin{eqnarray}
\mu_X\circ(\lambda\otimes \mathbbm 1) &=&  \epsilon \lambda\\ \label{eqn:integral}
 (\gamma \otimes \mathbbm 1)\circ \Delta_X &=& \eta \gamma. \label{eqn:cointegral}
\end{eqnarray}
Then we can define a Frobenius structure on $X$ with Frobenius pairing $\cup := \gamma \mu (\mathbbm 1\otimes S)$, and copairing $\cap := \Delta \lambda$. Such a construction also holds for finitely generated projective Hopf algebras (with trivial coinvariants), \cite{Par}. More generally, for an involutory Hopf monoid with (co)integrals in a symmetric monoidal category, under relatively mild assumptions, such as $\lambda \gamma = \mathbbm 1$ and $\lambda S \gamma = \mathbbm 1$ as in \cite{CD}, the same constructions as above are still valid. It was shown in \cite{SZ2} that under such circumstances, the quantum heap braiding of $X$ commutes with Frobenius paring and copairing. We summarize this in the following result. 

\begin{lemma}\label{lem:Frob}
	Let $X$ denote a (cocommutative) Hopf monoid such that there exist an integral $\lambda$ and a cointegral $\gamma$ such that setting $\cup = \gamma \mu (\mathbbm 1\otimes S)$ and $\cap = \Delta \lambda$, we obtain a Frobenius algebra structure on $X$. Then, we have the following equalities
	\begin{eqnarray*}
	\cup \otimes \mathbbm 1^{\otimes 2} &=& (\mathbbm 1^{\otimes 2}\otimes \cup) \circ c\\
	(\cup\otimes \mathbbm 1^{\otimes 2})\circ c &=& \mathbbm 1^{\otimes 2}\otimes \cup.
    \end{eqnarray*}
Similar conditions hold for $\cap$ and $c$. 
\end{lemma}
\begin{proof}
	The diagrammatic proof in \cite{SZ2}, Lemma~4.6, can be seen to apply in this case as well.  
\end{proof}

We can now define $\theta := (\Cup \otimes \mathbbm 1^{\otimes 2})\circ (\mathbbm 1^{\otimes 2}\otimes c)\circ (\Cap \otimes \mathbbm 1^{\otimes 2})$, where, by definition, we set $\Cup = \cup \circ (\mathbbm 1\otimes \cup \otimes \mathbbm 1)$, and a similar definition holds for $\Cap$. Consequently, it is easy to see (diagrammatically) that, applying Lemma~\ref{lem:Frob}, the braiding $c$ and $\theta$ commute, in the sense that $c\circ (\mathbbm 1\otimes \theta) = (\theta\otimes \mathbbm 1) \circ c$ and $c\circ  (\theta\otimes \mathbbm 1) = (\mathbbm 1\otimes \theta) \circ c$. In fact, although the proofs in \cite{SZ2} are for algebras in the category of modules, the diagrammatic proofs directly generalize to the case of Hopf-Frobenius algebras in more general symmetric monoidal categories, such as those of \cite{CD,BSZ}. Now, defining a category $\tilde R^*(X)$ as in Section~\ref{sec:ribboncat}, but replacing the twists with the new $\theta$ defined above and where the deforming $2$-cocycle is taken to be trivial, we have a proof analogous to that of Theorem~\ref{thm:ribbongeneral} implying that $\tilde R^*(X)$ is a ribbon category, since braiding and twist satisfy the same coherence properties of those of Theorem~\ref{thm:ribbongeneral}. 

Let now $\mathcal L$ be a framed link and let, in the same notation of Section~\ref{sec:ribbonquantum}, $b({\mathcal L}) = t_1^{k_1}\cdots t_n^{k_n} \tau$ be a framed braid whose closure is $\mathcal L$, where $\tau\in \mathbb B_n$. Decompose $\tau$ into a product of generators $\sigma_i$ of $\mathbb B_n$, namely $\tau = \sigma_{i_1}^{q_1} \cdots \sigma_{i_r}^{q_r}$ for some positive integers $r, q_1, \ldots q_r$ and $1\leq i_1, \ldots , i_r,\leq n-1$ for all $r$. We construct an endomorphism $\Phi_{\mathcal L}$ of $X^{\otimes 2n}$ assigning to each power of $\sigma_j$ appearing in the decomposition of $\tau$ into product of generators of $\mathbb B_n$, the morpshism $\mathbbm 1^{\otimes 2}\otimes \cdots \otimes  c_{2,2}^\alpha \otimes \cdots \otimes \mathbbm 1^{\otimes 2}$, where $c_{2,2}^\alpha$ is applied to the copies of $X^{\otimes 2}$ at positions $i$ and $i+1$. From the discussion in the previous paragraph it follows that $\Phi_{\mathcal L}$ so defined is an invariant of the link $\mathcal L$. Analogously, the quantum trace of $\Phi_{\mathcal L}$, $\Psi_{\mathcal L} := tr_q(\Phi_{\mathcal L})$ is an invariant of $\mathcal L$. 
\begin{remark}
	{\rm
Observe that such a construction of quantum invariants also applies to a Hopf-Frobenius structure where the twist is defined by means of pairing and copairing morphisms, upon exchanging $c^\alpha_{2,2}$ and $\theta_2^\alpha$ with $c$ and $\theta$ given above. 
}
\end{remark}
We conclude by briefly turning our attention to the categorical version of compatible TSD systems encountered in Section~\ref{sec:ribboncat}. In fact, due to a current lack of concrete examples, we have not explicitly mentioned such structures in Section~\ref{sec:generalized}. A possible source of examples, namely that of ternary augmented racks,  is discussed in the appendix below. The definitions of compatible systems translate the definitions given in Section~\ref{sec:ribboncat} in the setting of symmetric monoidal categories and allow to construct a braided monoidal category with certain base objects $\{X_i\}_{i\in I}$ generating the object family, and morpshims obtained inductively by combining all the possible compositions and tensoring of braidings $X_i\otimes X_i \otimes X_j\otimes X_j$. In order to obtain a quantum invariant of framed links we proceed as in the basic case given above, but replacing $\Phi_{\mathcal L}$ by a superposition of endomorphisms (hence we need to add the additional requirement of abelian monoidal category, or with $\mathbbm k$-linear structure), where the sum runs over all possible colorings $X_i\otimes X_i$ for each double string of the framed braid chosen to represent the link $\mathcal L$. 

\appendix

\section{Examples from ternary augmented racks}
It has been observed in Section~\ref{sec:ribboncat} that mutually distributive operations provide examples of compatible systems of self-distributive structures. The associated ribbon categories as in Theorem~\ref{thm:compatibleribboncat} have larger Hom sets than ribbon categories arising from a single ternary self-distributive structure, but consist of the same objects, i.e. a distinguished object $X^{\otimes 2}$ with all the tensor products arising from it. This class of examples still leaves open the question of whether there exist compatible systems with different base spaces, i.e. with $X_i \neq X_j$ for some $i$ and $j$. A possible answer comes from ternary augmented shelves \cite{ESZ}. We work in the category of vector spaces, but we observe that our example is obtained via linearization of a set-theoretic structure as given in Section~\ref{sec:ribboncat}. A related construction that is not obtained from linearized operations can be performed, in the setting of symmetric monoidal categories, by adapting the definitions in Section~\ref{sec:ribboncat} in a spirit similar to those given in Section\ref{sec:generalized}. We believe that this case is of interest, but unfortunately we are not aware of an example of such a ``non-linearized'' structure at this time. We have therefore decided to include it hereby as an appendix.

Suppose $X_1$ and $X_2$ are coaglebras and $H$ is a Hopf algebra that acts on them and, therefore, acts on $X_i^{\otimes 2}$ via comultiplication, for $i=1,2$. Suppose,  further, that there exist morphisms of coalgebras $p_i: X_i^{\otimes 2} \longrightarrow H$ satisfying 
$$
p_i(z\cdot \Delta(h)) = S(h^{(1)})p_i(z)h^{(2)},
$$
for every $z\in X_i$ and $h\in H$. Then the following result, extending Theorem B.6 in \cite{ESZ}, shows that in this situation a compatible system arises naturally.
\begin{theorem}\label{thm:compatiblehopf}
	Let $X_1$, $X_2$ and $H$ be as above, with related morphisms $p_1$ and $p_2$. Then, defining maps
	$$
	T_{ij} : X_i\otimes X_j\otimes X_j \longrightarrow X_i
	$$
	by extending the assignment $x\otimes y_1\otimes y_2 \mapsto x\cdot p_j(y_1\otimes y_2)$ by linearity, it follows that $\{X_i,T_{ij}\}_{i,j=1,2}$ is a compatible system of ternary self-distributive structures. A similar result holds for a finite number of $H$ modules $X_i$, $i=1,2\ldots , n$ with morphisms $p_i$. 
\end{theorem}
\begin{proof}
	The proof is analogous to that of Theorem B.6 in \cite{ESZ} and it is included here to show that having two morphisms $p_1$ and $p_2$ does not affect the procedure. It is enough to prove compatibility condition on simple tensors $x\otimes y_1\otimes y_2\otimes z_1\otimes z_2\in X_i\otimes X_j\otimes X_j\otimes X_k\otimes X_k$, where $i,j,k = 1,2$, as follows
	\begin{eqnarray*}
		\lefteqn{T_{ij}(T_{ik}(x\otimes z_1^{(1)}\otimes z_2^{(1)})\otimes T_{jk}(y_1\otimes z_1^{(2)}\otimes z_2^{(2)}\otimes T_{jk}(y_2\otimes z_1^{(3)}\otimes z_2^{(3)}))} \\
		&=& (x\cdot p_k(z_1^{(1)}\otimes z_2^{(1)}))\cdot p_i(y_1\cdot p_k(z_1^{(2)}\otimes z_2^{(2)})\otimes y_2\cdot p_k(z_1^{(3)}\otimes z_2^{(3)}))\\
		&=& x\cdot [p_k(z_1^{(1)}\otimes z_2^{(1)})\cdot p_j((y_1\otimes y_2)\cdot\Delta (p_k(z_1\otimes z_2)^{(2)})]\\
		&=& x\cdot [p_k(z_1\otimes z_2)^{(1)}\cdot S(p_k(z_1\otimes z_2)^{(2)})\cdot p_j(y_1\otimes y_2)\cdot p_k(z_1\otimes z_2)^{(3)}]\\
		&=& x\cdot [\epsilon(p_k(z_1\otimes z_2)^{(1)})1\cdot p_j(y_1\otimes y_2)\cdot p_k(z_1\otimes z_2)^{2}]\\
		&=& x\cdot [p_j(y_1\otimes y_2)\cdot p_k(z_1\otimes z_2)]\\
		&=& (x\cdot p_j(y_1\otimes y_2))\cdot p_k(z_1\otimes z_2)\\
		&=& T_{ik}(T_{ij}(x\otimes y_1\otimes y_2)\otimes z_1\otimes z_2),
	\end{eqnarray*}
	where in the second equality has been used the fact that $p_k$ is a morphism of coalgebras, the third equality is obtained applying the defining property of  $p_j$ and, finally, antipode axiom, counit axiom, associativity of the action of $H$ and definition of maps $T_{ij}$ are used in the last  four equalities. Coassociativity and associativity in $H$ are used throughout, without explicit mention. 
\end{proof}

\begin{example}\label{ex:compsyst}
	{\rm 
		Let $G_n$ denote the cyclic group of order $n$ in multiplicative notation, therefore an element of $G_n$ is denoted by the symbol $x^k$ for some $k$ determined by reduction modulo $n$. Let $m_1$ and $m_2$ be positive integers with $(m_1,m_2) = 1$ and define $G_{nm_1}$, $G_{nm_2}$ as above, with generators denoted by the symbols $y_1$ and $y_2$ respectively. Then the group algebra $H:= \mathbbm{k} [G_n]$ acts on $X_i:= \mathbbm{k} [G_{nm_i}]$ via the map $x^k \mapsto y_i^{m_ik}$ and multiplication in $G_{nm_i}$. Define maps $p_i : X_i\otimes X_i\longrightarrow H$ extending by linearity the assignment 
		$$
		y_i^{k_1}\otimes y_i^{k_2} \mapsto x^{k_2-k_1}.
		$$
		Then
		\begin{eqnarray*}
			p_i(y_i^{k_1}\otimes y_i^{k_2}\cdot \Delta(x^k)) &=&  	p_i(y_i^{nk+k_1}\otimes y_i^{nk+k_2})\\
			&=& x^{m_ik+k_2-m_ik-k_1} \\
			&=& x^{-k} x^{k_2-k_1}x^{k}\\
			&=& S(x^{k}) p_i(y_i^{k_1}\otimes y_i^{k_2}) x^k.
		\end{eqnarray*}
		Moreover, since $(m_1,m_2) = 1$, it follows that $X_1$ and $X_2$ are not subrepresentations of each others, so that the compatible system arising from Theorem~\ref{thm:compatiblehopf} consists of two base objects that are independent as representations of $H$. 
	}
\end{example}

\begin{example}
	{\rm 
		In this example it is shown that the compatible system given in Example~\ref{ex:compsyst} admits nontrivial compatible systems of $2$-cocycles. With same notation as before, fix $n = 2$, $m_1 = 2$ and $m_2 = 3$. The generator of $G_2$ is denoted by $x$, the ones of $G_4$ and $G_6$ by $y$ and $z$ respectively. The compatible structure on $X_1 := \mathbbm k [G_4]$, $X_2 := \mathbbm k[G_6]$ is given, on simple tensors, by 
		\begin{eqnarray*}
			T_{12} (y^t \otimes z^{s_1}\otimes z^{s_2}) &=& y^{t -2s_1 + 2s_2}\\
			T_{21}(z^s\otimes y^{t_1}\otimes y^{t_2}) &=& z^{s-3t_1+3t_2},
		\end{eqnarray*}
		while $T_{11}$ and $T_{22}$ are determined by linearizing the heap structure on groups $G_1$ and $G_2$, respectively. Define now maps $\alpha_{ij}: X_i\times X_j\times X_j \longrightarrow \Z$, where $\Z$ is given multiplicative structure with generator $w$, determined by 
		$$
		\alpha_{ij}(b_i^{k_i}\times b_j^{s_1}\times b_j^{s_2}) = \begin{cases}
		w \ \ {\rm if}\ \ s_1 = s_2\\
		1 \ \ {\rm otherwise} 
		\end{cases}
		$$
		where $b_1 = y$ and $b_2 = z$. The system so defined is seen to satisfy the $2$-cocycle condition of Section~\ref{sec:ribboncat} by direct inspection. In fact since $(y^k, y^t, y^t)$ and $(z^r, z^s, z^s)$ are $2$-cycles of $G_4$ and $G_6$, respectively, with ternary self-distributive differentials, applying Remark~\ref{rmk:nontrivialcocy} it follows that the system of $2$-cocycles $\{\alpha_{ij}\}$ is nontrivial. 
	}
\end{example}

	\end{document}